\documentclass[reqno]{amsart}
\usepackage{hyperref}
\usepackage{ amssymb }

\AtBeginDocument{%{\noindent\small
\vspace{9mm}}

\begin{document}
\title[3D Navier-Stokes equations]
{Existence, uniqueness and smoothness of a solution for 3D Navier-Stokes
equations with any smooth initial velocity. A priori estimate of this solution}

\author[A. Tsionskiy, M. Tsionskiy] 
{Arkadiy Tsionskiy, Mikhail Tsionskiy}

\subjclass[2010]{35Q30, 76D05.  
\newline
\newline
\textbf {Corresponding Authors address Email:} amtsionsk@gmail.com}

\begin{abstract}
 Solutions of the Navier-Stokes and Euler equations with initial conditions
 for 2D and 3D cases were obtained in the form of converging series, by
 an analytical iterative method using Fourier and Laplace transforms 
 \cite{TT10,TT11}.
 There the solutions are infinitely differentiable functions, and for
 several combinations of parameters numerical results are presented.
 The article provides a detailed proof of the existence, uniqueness and
 smoothness of the solution of the Cauchy problem for the 3D Navier-Stokes
 equations with any smooth initial velocity. When the viscosity tends to
 zero, this proof applies also to the Euler equations. A priori estimate 
 of this solution is presented.
\newline
\newline


\end{abstract}

\maketitle
\numberwithin{equation}{section}
\newtheorem{theorem}{Theorem}[section]
\newtheorem{remark}[theorem]{Remark}
\newtheorem{lemma}{Lemma} [section]
\allowdisplaybreaks

\section{Introduction}

Many authors have obtained results regarding the Euler
and Navier-Stokes equations.
Existence and smoothness of solution for the Navier-Stokes equations
in two dimensions have been known for a long time.
Leray (1934) showed that the Navier-Stokes equations in three dimensional
space have a weak solution. Scheffer (1976, 1993) and Shnirelman (1997)
obtained weak solution of the Euler equations with compact support in
space-time.  Caffarelli, Kohn and Nirenberg (1982) improved Scheffer's results,
and  Lin (1998) simplified the proof of the results by Leray.
Many problems and conjectures about behavior of weak solutions of the
Euler and Navier-Stokes equations are described in the books by
Ladyzhenskaya (1969), Temam (1977), Constantin (2001), Bertozzi and
Majda (2002), and Lemari\'e-Rieusset (2002).

The solution of the Cauchy problem for the 3D Navier-Stokes equations
is described in this article.
We will consider an initial velocity that is infinitely differentiable
and decreasing rapidly to zero in infinity. The applied force is assumed to be
identically zero. A  solution of the problem will be presented in the
following stages:

\textbf{First stage} (section 2).  We  move the non-linear parts of
 equations to the right side.
Then in section 4 we solve the system of linear partial differential equations with
constant coefficients. 

\textbf{Second stage} (section 3).  We  introduce perfect spaces of
functions and vector-functions (Gel'fand, Shilov \cite{GC68}),
in which we  look for the solution of the problem. We show the properties of the
direct and inverse Fourier transform for these functions.

\textbf{Third stage} (section 4, 5). We obtaine the solution of this system using
Fourier transforms for the space coordinates and Laplace transform for time.

From theorems about applications of Fourier and Laplace transforms, for system
of linear partial differential equations with constant coefficients, we
see that in this case if initial velocity and applied force are smooth enough
functions decreasing in infinity, then the solution of such system is also
a smooth function. Corresponding theorems are presented in  Bochner \cite{SB59},
 Palamodov \cite{VP70},  Shilov \cite{gS01},  Hormander \cite{LH83},
 Mizohata \cite{SM73},  Treves \cite{JFT61}.
The result of this stage is an integral equation for the vector-function
of velocity.

We  demonstrate the equivalence of solving the Cauchy problem
in differential form and in the form of an integral equation.

\textbf{Fourth stage} (section 6).
The properties of the matrix integral operators of the integral equation 
were obtained.

\textbf{Fifth stage} (section 7).
A priori estimate of the solution is presented by using the properties of  the matrix integral operators and the
direct and inverse Fourier transform.

\textbf{Sixth stage} (section 8).
The existence and uniqueness of the solution of
the Cauchy problem for the 3D Navier-Stokes equations is proved through the development of the ideas and approaches used to obtain a priori estimate of the solution.

\textbf{Seventh stage} (section 8).
By using a priori estimate of the
solution of the Cauchy problem for the 3D Navier-Stokes equations
\cite{oL69,LK63}, we show that the energy of the whole process has
a finite value for any $t$ in $[0,\infty)$.

\section{Mathematical setup}

The Navier-Stokes equations describe the motion of a fluid in
 $\mathbb{R}^N$ ($N = 3$). We look for a viscous incompressible fluid
filling all of $\mathbb{R}^N$  here. The Navier-Stokes equations are
then given by
\begin{gather}\label{eqn1}
\frac{\partial u_k}{\partial t} + \sum_{n=1}^N u_n
\frac{\partial u_k}{\partial x_n} = \nu\Delta u_k
- \frac{\partial p}{\partial x_k} +f_k(x,t)\quad
x\in \mathbb{R}^N,\;t\geq 0,\; 1\leq k \leq N\,, \\
\label{eqn2}
\operatorname{div}\vec{u}= \sum_{n=1}^N \frac{\partial u_n}{\partial x_n}
 = 0\quad  x\in \mathbb{R}^N,\;t\geq 0\,,
\end{gather}
with initial conditions
\begin{equation}\label{eqn3}
\vec{u}(x,0)= \vec{u}^0(x)\quad x\in \mathbb{R}^N\,.
\end{equation}
Here $\vec{u}(x,t)=(u_k(x,t)) \in \mathbb{R}^N$ $(1\leq k \leq N)$
is an unknown velocity vector,  $N = 3$;
 $p(x,t)$ is an unknown pressure; $\vec{u}^0(x)$ is a given
$C^{\infty}$ divergence-free vector field; $f_k(x,t)$ are components
of a given, externally applied force $\vec{f}(x,t)$; $\nu$ is a positive
 coefficient of the viscosity
(if $\nu = 0$ then
 \eqref{eqn1}--\eqref{eqn3} are the Euler equations);
and $ \Delta= \sum_{n=1}^N \frac{\partial^2}{\partial x_n^2}$
is the Laplacian in the space variables.
Equation \eqref{eqn1} is Newton's law for a fluid element.
Equation \eqref{eqn2} says that the fluid is incompressible.
For physically reasonable solutions, we accept
\begin{equation}\label{eqn4}
u_k(x,t) \to 0, \quad
\frac{\partial u_k}{\partial x_n} \to 0\quad
\text{as }| x | \to  \infty\quad 1\leq k \leq N ,\;  1\leq n \leq N\,.
\end{equation}
Hence, we will restrict our attention to initial conditions
 $\vec{u}^0$ and force $\vec{f}$ that satisfy
\begin{equation}\label{eqn5}
|\partial_{x}^{\alpha}\vec{u}^0(x)|\leq C_{\alpha K}(1+| x |)^{-K}
\quad \text{on $\mathbb{R}^N$ for any $\alpha$ and any $K$}.
\end{equation}
and
\begin{equation}\label{eqn6}
|\partial_{x}^{\alpha}\partial_{t}^{\beta}\vec{f}(x,t)|
\leq C_{\alpha \beta K}(1+| x | +t)^{-K} \quad
\text{on $\mathbb{R}^N\times[0,\infty)$ for any $\alpha,\beta$  and any $K$}.
\end{equation}
\[
C_{\alpha K},\;\; C_{\alpha \beta K}\;\; - constants.
\]
To start the process of solution let us add
$- \sum_{n=1}^N u_n\frac{\partial u_k}{\partial x_n}$ to both sides of
the equations \eqref{eqn1}. Then we have
\begin{gather}\label{eqn7}
\frac{\partial u_k}{\partial t}=\nu\,\Delta\,u_k
-\frac{\partial p}{\partial x_k}+f_k(x,t)-\sum_{n=1}^N
u_n\frac{\partial u_k}{\partial x_n}\quad x\in \mathbb{R}^N,\;t\geq 0,\; 1\leq k \leq N,\\
\label{eqn8}
\operatorname{div}\vec{u}= \sum_{n=1}^N \frac{\partial u_n}{\partial x_n} = 0
\quad  x\in \mathbb{R}^N,\;t\geq 0,\\
\label{eqn9}
\vec{u}(x,0)= \vec{u}^0(x)\quad x\in \mathbb{R}^N,\\
\label{eqn10}
u_k(x,t) \to 0\quad \frac{\partial u_k}{\partial x_n}\to 0\quad\text{as }
 | x | \to  \infty\quad 1\leq k \leq N,\; 1\leq n \leq N,\\
\label{eqn11}
|\partial_{x}^{\alpha}\vec{u}^0(x)|\leq C_{\alpha K}(1+| x |)^{-K} \quad
\text{on $\mathbb{R}^N$ for  any $\alpha$ and any $K$},\\
\label{eqn12}
|\partial_{x}^{\alpha}\partial_{t}^{\beta}\vec{f}(x,t)|
\leq C_{\alpha \beta K}(1+| x | +t)^{-K} \quad
\text{on $\mathbb{R}^N\times[0,\infty)$  for  any $\alpha,\beta$  and any $K$}.
\end{gather}
Let us denote
\begin{equation}\label{eqn13}
\tilde{f}_k(x,t)= f_k(x,t)
- \sum_{n=1}^N u_n\frac{\partial u_k}{\partial x_n}\quad 1\leq k \leq N\,.
\end{equation}
We can present it in the vector form as
\begin{equation}\label{eqn14}
\vec{\tilde{f}}(x,t)= \vec{f}(x,t) -(\vec{u}\cdot \nabla)\vec{u}\,.
\end{equation}

\section{Spaces $S$ and  $\stackrel{\longrightarrow} {TS}$.\\
Fourier transforms in  Space $S$.} 

As in \cite{GC68, RR64}, we consider the space $S$ (Schwartz) of all infinitely
differentiable functions  $\varphi$(x) defined in $N$-dimensional
space $\mathbb{R}^N$ ($N = 3$), such that these functions tend to 0 as
$| x | \to  \infty$, as well as their derivatives of any order,
more rapidly than any power of $1/|x|$.

To define a topology in the space $S$ let us introduce countable system of norms
\begin{equation}\label{eqn200}
\|\varphi\|_{p}= \sup_{x}\big\{ | x^{k}D^{q} \varphi(x)| ,\, 0\leq k\leq p, \;\;
0\leq q \leq p\big\}
\quad p = 0, 1, 2,\dots,
\end{equation}
where
\begin{gather*}
| x^{k}D^{q} \varphi(x)|= | x_1^{k_1}\dots x_{N}^{k_{N}}
\frac{\partial^{q_1+\dots + q_{N}}\varphi(x)}{\partial {x_1^{q_1}}
\dots \partial {x_{N}^{q_{N}}} }|,\\
k = (k_1, \dots , k_{N}),\quad
q = (q_1, \dots , q_{N}),\quad
x^{k} =  x_1^{k_1}\dots x_{N}^{k_{N}} ,\\
D^{q} = \frac{\partial^{q_1+\dots + q_{N}}}{\partial {x_1^{q_1}}\dots
 \partial {x_{N}^{q_{N}}} },\quad  q_1, \dots , q_{N} = 0, 1, 2, \dots\,.
\end{gather*}
\begin{equation}\label{eqn201z9b}
\|\varphi\|_{0}\leq\|\varphi\|_{1}\leq\cdots\leq\|\varphi\|_{p}\cdots\\
\end{equation}
Note that $S$ is a perfect space. %(complete countably normed space,
%in which the bounded sets are compact).
The space $\overrightarrow{TS}$ of vector-functions $\vec{\varphi}$ is a
direct sum of $N$  perfect spaces $S$ ($N = 3$) \cite{VT80}:
\[
\overrightarrow{TS} = S \oplus S \oplus S.
\]
To define a topology in the space $\overrightarrow{TS}$ let us introduce
countable system of norms
\begin{equation}\label{eqn201}
\|\vec{\varphi}\|_{p}= \sum_{i = 1}^N\|\varphi_{i}\|_{p}
 = \sum_{i = 1}^N\sup_{x}\big\{ | x^{k}D^{q} \varphi_{i}(x)|,\,
0\leq k\leq p, \;\;0\leq q \leq p\big\},
\end{equation}
\[N = 3, p = 0, 1, 2,\dots\]
\begin{equation}\label{eqn201z9c}
\|\vec{\varphi}\|_{0}\leq\|\vec{\varphi}\|_{1}\leq\cdots\leq\|\vec{\varphi}\|_{p}\cdots
\end{equation}
Let us consider the Fourier transform of the function $\varphi(x) \in S$ \cite{GC68}. 

We show that the Fourier transform of the function $\varphi(x)$   
\begin{equation}\label{eqn201z1}
F[\varphi] \equiv \psi(\sigma) \equiv \widetilde{\varphi(x)} \equiv \frac{1}{(2\pi)^{N/2}}\int_{\mathbb{R}^N} \text{\LARGE{e}}^{i(x,\sigma)}\varphi(x)dx,  \;\;\;\;\;\;\;\;
(x,\sigma) = \sum_{i = 1}^Nx_i\sigma_i,
\end{equation}
also belongs, as a function of $\sigma$, to the space $S$ (a function of $\sigma$), i.e., $\psi(\sigma)$ is  infinitely
differentiable, and each of its derivatives approaches zero more rapidly than any power of $1/|\sigma|$ as $|\sigma| \rightarrow \infty$.

The integral in  \eqref{eqn201z1} admits of differentiation with respect to the parameter $\sigma_j$, since the integral obtained after formal differentiation remains absolutely convergent:
\[
\frac{\partial\psi(\sigma)}{\partial\sigma_j} = \frac{1}{(2\pi)^{N/2}}\int_{\mathbb{R}^N} ix_j\text{\LARGE{e}}^{i(x,\sigma)}\varphi(x)dx
\]
The properties of the function $\varphi(x)$  permit this differentiation to be continued without limit. This means that \textit{the function $\psi(\sigma)$ is infinitely differentiable}. Hence, the following formula holds
\begin{equation}\label{eqn201z2}
P(D)F[\varphi(x)] \equiv P(D)\psi(\sigma) = \frac{1}{(2\pi)^{N/2}}\int_{\mathbb{R}^N} P(ix)\text{\LARGE{e}}^{i(x,\sigma)}\varphi(x)dx =   F[P(ix)\varphi(x)]
\end{equation}
for any differential operator $P(D)$:
\[
P(D) =  \sum a_kD^k = \sum a_{k_{1}...k_{n}}\frac{\partial^{k_1+\dots + k_{n}}}{\partial {\sigma_1^{k_1}}\dots
 \partial {\sigma_{n}^{k_{n}}} };
\]
similarly
\[
P(ix) = \sum a_k(ix)^k = \sum a_{k_{1}...k_{n}}(ix_1)^{k_1}\dots(ix_n)^{k_n}.
\]
Now, let us consider the Fourier transform of the partial derivative ($\partial\varphi/\partial x_j$):
\[
F\Big[\frac{\partial\varphi(x)}{\partial x_j}\Big] = \frac{1}{(2\pi)^{N/2}}\int_{\mathbb{R}^N} \frac{\partial\varphi(x)}{\partial x_j}\text{\LARGE{e}}^{i(x,\sigma)}dx.
\]
Integration by parts, taking into account that $\varphi(x)$ tends to zero as $|x| \rightarrow \infty$, leads to the expression
\[
F\Big[\frac{\partial\varphi(x)}{\partial x_j}\Big] = -i\sigma_j \frac{1}{(2\pi)^{N/2}}\int_{\mathbb{R}^N} \varphi(x)\text{\LARGE{e}}^{i(x,\sigma)}dx = -i\sigma_j F[\varphi(x)].
\]
Repeating this operation we obtain
\begin{equation}\label{eqn201z3}
F[P(D)\varphi(x)] = P(-i\sigma_j)F[\varphi(x)].
\end{equation}
As a Fourier transform of an integrable function, the function $P(-i\sigma_j)F[\varphi(x)]$ is bounded. Since $P$ is any polynomial, we see that $F[\varphi(x)] = \psi(\sigma)$ \textit{tends to zero more rapidly than any power of $1/|\sigma|$ as $|\sigma|\rightarrow \infty$}. The same is true also for any derivative of $\psi(\sigma)$ since, as we have seen, the expression $\partial\psi/\partial\sigma_j$ say, is the Fourier transform of the function $ix_j\varphi(x)$, which  also belongs to $S$.

Therefore, any derivative of $\psi(\sigma)$ tends to zero more rapidly than any power of $1/|\sigma|$ as $|\sigma| \rightarrow \infty$, Q.E.D.

Thus, \textit{if a function $\varphi(x)$ belongs to the space S (a function of $x$), then $\psi(\sigma) = F[\varphi(x)]$ also belongs to the space S (a function of $\sigma$).}

An analogous statement is proved in exactly the same manner for the inverse Fourier transform $F^{-1}$, which, as is known, is defined by the formula
\begin{equation}\label{eqn201z4}
\varphi(x) = F^{-1}[\psi(\sigma)]  = \frac{1}{(2\pi)^{N/2}}\int_{\mathbb{R}^N} \text{\LARGE{e}}^{-i(x,\sigma)}\psi(\sigma)d\sigma;
\end{equation} 
\textit{if $\psi(\sigma)$ belongs to the space S (a function of $\sigma$), then $\varphi(x) = F^{-1}[\psi(\sigma)]$ also belongs to the space S (a function of $x$).}

Let us note that by applying the operator $F^{-1}$ to \eqref{eqn201z2} and \eqref{eqn201z3}, and replacing $F[\varphi]$ everywhere by $\psi$, and $\varphi$ by $F^{-1}[\psi]$, we obtain the following formulas for the operator $F^{-1}$:
\begin{equation}\label{eqn201z5}
F^{-1}[P(D)\psi(\sigma)] = P(ix)F^{-1}[\psi(\sigma)];
\end{equation}
\begin{equation}\label{eqn201z6}
P(D)F^{-1}[\psi(\sigma)] = F^{-1}[P(-i\sigma)\psi(\sigma)].
\end{equation}
From the proved assumptions, it follows that \textit{the operators $F$ and $F^{-1}$ map the space $S$ conformally one-to-one into itself.} These operators are evidently linear.

We introduce the infinitely differentiable function:
\begin{equation}\label{eqn201aa}
\delta (\gamma_1,\gamma_2,\gamma_3) = \text{\LARGE{e}}  ^{-\frac{\displaystyle\epsilon^3}{(\gamma_1^2 + \gamma_2^2 + \gamma_3^2)}}\;,\;\; 0 < \epsilon << 1.
\end{equation}

For example $\epsilon = \text{e}^{-q_{1}}, \;\;\;q_{1} = 2, 3,4,...\;\;\;q_{1} < \infty.$

It is evident that
\begin{equation}\label{eqn201zz}
\lim_{\gamma_1,\gamma_2,\gamma_3\rightarrow 0} \frac{1}{(\gamma_1^2 + \gamma_2^2 + \gamma_3^2)^n}\cdot \text{\LARGE{e}}  ^{-\frac{\displaystyle\epsilon^3}{(\gamma_1^2 + \gamma_2^2 + \gamma_3^2)}} = 0
\end{equation}
for any $0 \leq n < \infty $.

\section{Solution of the system \eqref{eqn7}--\eqref{eqn14}}

We seek a solution of the system \eqref{eqn7}--\eqref{eqn14}:

$\vec{u}(x_1,x_2,x_3,t) \in \overrightarrow{TS}$, $p(x_1,x_2,x_3,t) \in S.$ 

$\vec{u}_0(x_1,x_2,x_3)\in \overrightarrow{TS}\;\; \text{and}\;\; \vec{\tilde{f}}(x_1,x_2,x_3,t) \in \overrightarrow{TS}$ also.

Let us assume that all operations below are valid.
The validity of these operations will be proved in the next sections.
Taking into account our substitution \eqref{eqn13} we see that
\eqref{eqn7}--\eqref{eqn9} are in fact system of linear partial
differential equations with constant coefficients.

The solution of this system will be presented by the following steps:

\textbf{First step.} We  use Fourier transform \eqref{A3} to solve equations
\eqref{eqn7}--\eqref{eqn14}.
We obtain:
\begin{gather*}
U_k( \gamma_1, \gamma_2, \gamma_3, t)=F[u_k ( x_1, x_2, x_3, t)],\\
 -\gamma^2_{s}U_k( \gamma_1, \gamma_2, \gamma_3, t)
=F[ \frac{\partial^2u_k( x_1, x_2, x_3, t)}{\partial x^2_{s}}]
\quad \text{(use \eqref{eqn10})},\\
U_k^0( \gamma_1 ,\gamma_2 ,\gamma_3)=F[u_k^0 ( x_1, x_2, x_3)],\\
P( \gamma_1, \gamma_2, \gamma_3, t)=F[p ( x_1, x_2, x_3, t)],\\
\tilde{F}_k( \gamma_1, \gamma_2, \gamma_3, t)
=F[\tilde{f}_k ( x_1, x_2, x_3, t)],
\end{gather*}
for $k,s=1,2,3$. Then
\begin{gather}\label{eqn134}
\begin{aligned}
\frac{d U_1( \gamma_1, \gamma_2, \gamma_3, t )}{d t}
& =-\nu
( \gamma_1^2 +\gamma_2^2 +\gamma_3^2) U_1( \gamma_1, \gamma_2,
\gamma_3, t )+i\gamma_1 P( \gamma_1, \gamma_2, \gamma_3, t )\\
&\quad + \tilde{F}_1( \gamma_1, \gamma_2, \gamma_3,t ),
\end{aligned} \\
\label{eqn135}
\begin{aligned}
\frac{d U_2( \gamma_1, \gamma_2, \gamma_3, t )}{d t}
&=-\nu ( \gamma_1^2 +\gamma_2^2 +\gamma_3^2) U_2( \gamma_1,
 \gamma_2, \gamma_3, t )+i\gamma_2 P( \gamma_1, \gamma_2, \gamma_3, t )\\
&\quad + \tilde{F}_2( \gamma_1, \gamma_2, \gamma_3,t ),
\end{aligned} \\
\label{eqn136}
\begin{aligned}
\frac{d U_3( \gamma_1, \gamma_2, \gamma_3, t )}{d t}
&=-\nu ( \gamma_1^2 +\gamma_2^2 +\gamma_3^2)
U_3( \gamma_1, \gamma_2, \gamma_3, t )
+i\gamma_3 P( \gamma_1, \gamma_2, \gamma_3, t )\\
&\quad + \tilde{F}_3( \gamma_1, \gamma_2, \gamma_3,t ),
\end{aligned}\\
\label{eqn137}
\gamma_1 U_1( \gamma_1, \gamma_2, \gamma_3, t )
+ \gamma_2\, U_2( \gamma_1, \gamma_2, \gamma_3, t )
+ \gamma_3\, U_3( \gamma_1, \gamma_2, \gamma_3, t ) =0,\\
\label{eqn138}
U_1(\gamma_1, \gamma_2,  \gamma_3,  0)= U_1^0(\gamma_1 ,\gamma_2 ,\gamma_3),\\
\label{eqn139}
U_2(\gamma_1, \gamma_2,  \gamma_3,  0)= U_2^0(\gamma_1 ,\gamma_2 ,\gamma_3),\\
\label{eqn140}
U_3(\gamma_1, \gamma_2,  \gamma_3,  0)= U_3^0(\gamma_1 ,\gamma_2 ,\gamma_3)\,.
\end{gather}

Hence, we have received a system of linear ordinary differential equations
with constant coefficients \eqref{eqn134}-\eqref{eqn140}
 according to Fourier transforms. At the same time the initial conditions
are set only for Fourier transforms of velocity components
$U_1( \gamma_1, \gamma_2, \gamma_3, t )$,
$U_2( \gamma_1, \gamma_2, \gamma_3, t )$,
$U_3( \gamma_1, \gamma_2, \gamma_3, t )$.
Because of that we can eliminate Fourier transform for pressure
$P( \gamma_1, \gamma_2, \gamma_3, t )$ from equations
\eqref{eqn134}--\eqref{eqn136}
on the next step of the solution process.

\textbf{Second step.}
From here assuming that $\gamma_1 \neq 0, \gamma_2 \neq 0,\gamma_3 \neq 0$

(case $\gamma_1 = \gamma_2 = \gamma_3 = 0$  will be considered later in this article),
 we eliminate $P( \gamma_1, \gamma_2, \gamma_3, t )$ from equations
$\eqref{eqn134}- \eqref{eqn136}$ and find
\begin{gather}\label{eqn141}
\begin{aligned}
&\frac{d}{dt} [ U_2( \gamma_1, \gamma_2, \gamma_3, t )
- \frac{\gamma_2}{\gamma_1} \,U_1( \gamma_1, \gamma_2, \gamma_3, t) ]\\
&= -\nu( \gamma_1^2 +\gamma_2^2 +\gamma_3^2)
[ U_2( \gamma_1, \gamma_2, \gamma_3, t )
- \frac{\gamma_2}{\gamma_1} U_1( \gamma_1, \gamma_2, \gamma_3, t) ]\\
&\quad +  [ \tilde{F}_2( \gamma_1, \gamma_2, \gamma_3, t )
-\frac{\gamma_2}{\gamma_1}  \tilde{F}_1( \gamma_1, \gamma_2, \gamma_3, t) ]\,,
\end{aligned}\\
\label{eqn142}
\begin{aligned}
&\frac{d}{dt} [  U_3( \gamma_1, \gamma_2, \gamma_3, t )
 -\frac{\gamma_3}{\gamma_1}\, U_1( \gamma_1, \gamma_2, \gamma_3, t) ] \\
&= -\nu( \gamma_1^2 +\gamma_2^2 +\gamma_3^2)
[ U_3( \gamma_1, \gamma_2, \gamma_3, t )
- \frac{\gamma_3}{\gamma_1}\, U_1( \gamma_1, \gamma_2, \gamma_3, t) ]\\
&\quad  + [ \tilde{F}_3( \gamma_1, \gamma_2, \gamma_3, t )
-\frac{\gamma_3}{\gamma_1} \tilde{F}_1( \gamma_1, \gamma_2, \gamma_3, t) ]\,,
\end{aligned}\\
\label{eqn143}
\gamma_1 U_1( \gamma_1, \gamma_2, \gamma_3, t )
+ \gamma_2\, U_2( \gamma_1, \gamma_2, \gamma_3, t )
+ \gamma_3\, U_3( \gamma_1, \gamma_2, \gamma_3, t ) =0,\\
\label{eqn144}
U_1(\gamma_1, \gamma_2,  \gamma_3,  0)= U_1^0(\gamma_1 ,\gamma_2 ,\gamma_3),\\
\label{eqn145}
U_2(\gamma_1, \gamma_2,  \gamma_3,  0)= U_2^0(\gamma_1 ,\gamma_2 ,\gamma_3),\\
\label{eqn146}
U_3(\gamma_1, \gamma_2,  \gamma_3,  0)= U_3^0(\gamma_1 ,\gamma_2 ,\gamma_3)\,.
\end{gather}

\textbf{Third step.} We use Laplace transform \eqref{A4}, \eqref{A5}
for a system of linear ordinary differential equations with constant
coefficients \eqref{eqn141}--\eqref{eqn143}  and have as a result the
system of linear algebraic equations with constant coefficients:
\begin{gather}
U_k^{\otimes} (\gamma_1, \gamma_2, \gamma_3, \eta)
=L[U_k(\gamma_1, \gamma_2, \gamma_3, t)] \quad   k=1,2,3;\\
\tilde{F}_k^{\otimes} (\gamma_1, \gamma_2, \gamma_3, \eta)
=L[\tilde{F}_k(\gamma_1, \gamma_2, \gamma_3, t)] \quad   k=1,2,3; \\
\label{eqn147}
\begin{aligned}
&\eta [  U_2^{\otimes}( \gamma_1, \gamma_2, \gamma_3, \eta )
 -\frac{\gamma_2}{\gamma_1} U_1^{\otimes}( \gamma_1, \gamma_2, \gamma_3, \eta) ]\\
&\quad - [   U_2( \gamma_1, \gamma_2, \gamma_3, 0 )
 -\frac{\gamma_2}{\gamma_1} U_1( \gamma_1, \gamma_2, \gamma_3, 0) ]  \\
&= -\nu( \gamma_1^2 +\gamma_2^2 +\gamma_3^2)[   U_2^{\otimes}
( \gamma_1, \gamma_2, \gamma_3, \eta )
 -\frac{\gamma_2}{\gamma_1} U_1^{\otimes}( \gamma_1, \gamma_2, \gamma_3, \eta) ]\\
&\quad + [   \tilde{F}_2^{\otimes}( \gamma_1, \gamma_2, \gamma_3, \eta )
 -\frac{\gamma_2}{\gamma_1} \tilde{F}_1^{\otimes}
( \gamma_1, \gamma_2, \gamma_3, \eta) ],
\end{aligned}\\
\label{eqn148}
\begin{aligned}
&\eta  [   U_3^{\otimes}( \gamma_1, \gamma_2, \gamma_3, \eta )
-\frac{\gamma_3}{\gamma_1} U_1^{\otimes}( \gamma_1, \gamma_2, \gamma_3, \eta) ]\\
&\quad - [ U_3( \gamma_1, \gamma_2, \gamma_3, 0 ) -\frac{\gamma_3}{\gamma_1}
U_1( \gamma_1, \gamma_2, \gamma_3, 0) ] \\
&= -\nu( \gamma_1^2 +\gamma_2^2 +\gamma_3^2)[ U_3^{\otimes}
 ( \gamma_1, \gamma_2, \gamma_3, \eta )
 -\frac{\gamma_3}{\gamma_1} U_1^{\otimes}( \gamma_1, \gamma_2, \gamma_3, \eta) ]\\
&\quad + [    \tilde{F}_3^{\otimes}( \gamma_1, \gamma_2, \gamma_3, \eta )
 -\frac{\gamma_3}{\gamma_1} \tilde{F}_1^{\otimes}( \gamma_1, \gamma_2, \gamma_3,
 \eta) ],
\end{aligned} \\
\label{eqn149}
\gamma_1 U_1^{\otimes}( \gamma_1, \gamma_2, \gamma_3, \eta )
+ \gamma_2\, U_2^{\otimes}( \gamma_1, \gamma_2, \gamma_3, \eta )
+ \gamma_3\, U_3^{\otimes}( \gamma_1, \gamma_2, \gamma_3, \eta ) =0,\\
\label{eqn150}
U_1(\gamma_1, \gamma_2,  \gamma_3,  0)= U_1^0(\gamma_1 ,\gamma_2 ,\gamma_3),\\
\label{eqn151}
U_2(\gamma_1, \gamma_2,  \gamma_3,  0)= U_2^0(\gamma_1 ,\gamma_2 ,\gamma_3),\\
\label{eqn152}
U_3(\gamma_1, \gamma_2,  \gamma_3,  0)= U_3^0(\gamma_1 ,\gamma_2 ,\gamma_3)\,.
\end{gather}

Let us rewrite the system of equations \eqref{eqn147}--\eqref{eqn149}
in the form
\begin{gather}\label{eqn147a}
\begin{aligned}
&[\eta +\nu( \gamma_1^2 +\gamma_2^2 +\gamma_3^2)]
 \frac{\gamma_2}{\gamma_1} U_1^{\otimes}( \gamma_1, \gamma_2, \gamma_3, \eta)\\
&\quad - [\eta +\nu( \gamma_1^2 +\gamma_2^2 +\gamma_3^2)] U_2^{\otimes}
( \gamma_1, \gamma_2, \gamma_3, \eta)  \\
&= [  \frac{\gamma_2}{\gamma_1} \tilde{F}_1^{\otimes}( \gamma_1, \gamma_2,
\gamma_3, \eta)  -\tilde{F}_2^{\otimes}( \gamma_1, \gamma_2, \gamma_3, \eta ) ]\\
&\quad + [ \frac{\gamma_2}{\gamma_1} U_1( \gamma_1, \gamma_2, \gamma_3, 0)
 -U_2( \gamma_1, \gamma_2, \gamma_3, 0 ) ],
\end{aligned} \\
\label{eqn148a}
\begin{aligned}
&[\eta +\nu( \gamma_1^2 +\gamma_2^2 +\gamma_3^2)]
 \frac{\gamma_3}{\gamma_1} U_1^{\otimes}( \gamma_1, \gamma_2, \gamma_3, \eta)\\
 &\quad - [\eta +\nu( \gamma_1^2 +\gamma_2^2 +\gamma_3^2)] U_3^{\otimes}
( \gamma_1, \gamma_2, \gamma_3, \eta)  \\
&=[  \frac{\gamma_3}{\gamma_1} \tilde{F}_1^{\otimes}( \gamma_1, \gamma_2,
\gamma_3,
 \eta)  -\tilde{F}_3^{\otimes}( \gamma_1, \gamma_2, \gamma_3, \eta ) ]\\
&\quad + [ \frac{\gamma_3}{\gamma_1} U_1( \gamma_1, \gamma_2, \gamma_3, 0)
 -U_3( \gamma_1, \gamma_2, \gamma_3, 0 ) ],
\end{aligned}\\
\label{eqn149a}
\gamma_1 U_1^{\otimes}( \gamma_1, \gamma_2, \gamma_3, \eta )
+ \gamma_2\, U_2^{\otimes}( \gamma_1, \gamma_2, \gamma_3, \eta )
 + \gamma_3\, U_3^{\otimes}( \gamma_1, \gamma_2, \gamma_3, \eta ) =0
\end{gather}

The determinant of this system is
\begin{equation}\label{eqn1481b}
\begin{aligned}
\Delta &= \begin{vmatrix}
[\eta +\nu( \gamma_1^2 +\gamma_2^2 +\gamma_3^2)]\frac{\gamma_2}{\gamma_1}
& -[\eta +\nu( \gamma_1^2 +\gamma_2^2 +\gamma_3^2)] & 0 \\
 [\eta +\nu( \gamma_1^2 +\gamma_2^2 +\gamma_3^2)]\frac{\gamma_3}{\gamma_1}
& 0 & -[\eta +\nu( \gamma_1^2 +\gamma_2^2 +\gamma_3^2)] \\
\gamma_1 & \gamma_2 & \gamma_3 \end{vmatrix} \\
& = \frac{[\eta +\nu( \gamma_1^2 +\gamma_2^2
+\gamma_3^2)]^2( \gamma_1^2 +\gamma_2^2
+\gamma_3^2)}{\gamma_1} \neq  0\,.
\end{aligned}
\end{equation}
Consequently the system of equations \eqref{eqn147}--\eqref{eqn149}
 and/or \eqref{eqn147a}--\eqref{eqn149a} has a unique solution.
Taking into account formulas \eqref{eqn150}--\eqref{eqn152}
 we can write this solution in the form
\begin{gather}\label{eqn153}
\begin{aligned}
&U_1^{\otimes}( \gamma_1, \gamma_2, \gamma_3, \eta )\\
&=\frac{[( \gamma_2^2 +\gamma_3^2)  \tilde{F}_1^{\otimes}
( \gamma_1, \gamma_2, \gamma_3, \eta)
- \gamma_1\gamma_2 \tilde{F}_2^{\otimes}( \gamma_1, \gamma_2, \gamma_3, \eta)
- \gamma_1\gamma_3 \tilde{F}_3^{\otimes}( \gamma_1, \gamma_2, \gamma_3, \eta)]}
{ (\gamma_1^2 +\gamma_2^2 +\gamma_3^2)
 [\eta+\nu (\gamma_1^2 +\gamma_2^2 +\gamma_3^2)] }\\
&\quad + \frac{ U_1^0(\gamma_1 , \gamma_2 , \gamma_3)}
{[\eta+\nu (\gamma_1^2 +\gamma_2^2 +\gamma_3^2)] }\,,
\end{aligned}\\
\label{eqn154}
\begin{aligned}
&U_2^{\otimes}( \gamma_1, \gamma_2, \gamma_3, \eta )\\
&=\frac{[( \gamma_3^2 +\gamma_1^2)  \tilde{F}_2^{\otimes}
 ( \gamma_1, \gamma_2, \gamma_3, \eta) - \gamma_2\gamma_3 \tilde{F}_3^{\otimes}
( \gamma_1, \gamma_2, \gamma_3, \eta) - \gamma_2\gamma_1 \tilde{F}_1^{\otimes}
 ( \gamma_1, \gamma_2, \gamma_3, \eta)]}{ (\gamma_1^2 +\gamma_2^2
 +\gamma_3^2) [\eta+\nu (\gamma_1^2 +\gamma_2^2 +\gamma_3^2)] }\\
&\quad + \frac{ U_2^0(\gamma_1 , \gamma_2 , \gamma_3)}{[\eta+\nu (\gamma_1^2
+\gamma_2^2 +\gamma_3^2)] }\,,
\end{aligned} \\
\label{eqn155}
\begin{aligned}
&U_3^{\otimes}( \gamma_1, \gamma_2, \gamma_3, \eta )\\
&=\frac{[( \gamma_1^2 +\gamma_2^2)
\tilde{F}_3^{\otimes}( \gamma_1, \gamma_2, \gamma_3, \eta)
- \gamma_3\gamma_1 \tilde{F}_1^{\otimes}( \gamma_1, \gamma_2, \gamma_3, \eta)
 - \gamma_3\gamma_2 \tilde{F}_2^{\otimes}( \gamma_1, \gamma_2, \gamma_3, \eta)]}
{ (\gamma_1^2 +\gamma_2^2 +\gamma_3^2)
 [\eta+\nu (\gamma_1^2 +\gamma_2^2 +\gamma_3^2)] }\\
&\quad + \frac{ U_3^0(\gamma_1 , \gamma_2 , \gamma_3)}
{[\eta+\nu (\gamma_1^2 +\gamma_2^2 +\gamma_3^2)] }\,.
\end{aligned}
\end{gather}
Then we use the convolution theorem with the convolution formula
\eqref{A6} and integral \eqref{A7} for \eqref{eqn153}--\eqref{eqn155}
to obtain 
\begin{gather}\label{eqn156}
\begin{aligned}
&U_1(\gamma_1, \gamma_2, \gamma_3, t)\\
&=\int_0^t e ^{-\nu (\gamma_1^2 +\gamma_2^2
+\gamma_3^2) (t-\tau)}\\
 &\quad\times \frac{[( \gamma_2^2 +\gamma_3^2)
 \tilde{F}_1( \gamma_1, \gamma_2, \gamma_3, \tau)
-\gamma_1\gamma_2 \tilde{F}_2( \gamma_1, \gamma_2, \gamma_3, \tau )
 -\gamma_1\gamma_3 \tilde{F}_3( \gamma_1, \gamma_2, \gamma_3, \tau ) ]}
{ (\gamma_1^2 +\gamma_2^2+\gamma_3^2 ) }\,d\tau \\
&\quad +e ^{-\nu (\gamma_1^2 +\gamma_2^2 +\gamma_3^2) t}
 U_1^0(\gamma_1 ,\gamma_2 ,\gamma_3),
\end{aligned}\\
\label{eqn157}
\begin{aligned}
&U_2(\gamma_1, \gamma_2, \gamma_3, t)\\
&=\int_0^t e ^{-\nu (\gamma_1^2 +\gamma_2^2 +\gamma_3^2) (t-\tau)} \\
&\quad\times
\frac{[( \gamma_3^2 +\gamma_1^2)
 \tilde{F}_2( \gamma_1, \gamma_2, \gamma_3, \tau)
 -\gamma_2\gamma_3 \tilde{F}_3( \gamma_1, \gamma_2, \gamma_3, \tau )
 -\gamma_2\gamma_1 \tilde{F}_1( \gamma_1, \gamma_2, \gamma_3, \tau ) ]}
{ (\gamma_1^2 +\gamma_2^2+\gamma_3^2 ) }\,d\tau \\
&\quad +e ^{-\nu (\gamma_1^2 +\gamma_2^2 +\gamma_3^2) t}
U_2^0(\gamma_1 ,\gamma_2 ,\gamma_3),
\end{aligned} \\
\label{eqn158}
\begin{aligned}
&U_3(\gamma_1, \gamma_2, \gamma_3, t)\\
&=\int_0^t e ^{-\nu (\gamma_1^2 +\gamma_2^2 +\gamma_3^2) (t-\tau)} \\
&\quad\times \frac{[( \gamma_1^2 +\gamma_2^2)
\tilde{F}_3( \gamma_1, \gamma_2, \gamma_3, \tau) -\gamma_3\gamma_1
\tilde{F}_1( \gamma_1, \gamma_2, \gamma_3, \tau ) -\gamma_3\gamma_2 \tilde{F}_2
( \gamma_1, \gamma_2, \gamma_3, \tau ) ]}{ (\gamma_1^2
+\gamma_2^2+\gamma_3^2 ) }\,d\tau \\
&\quad +e ^{-\nu (\gamma_1^2 +\gamma_2^2 +\gamma_3^2) t}
U_3^0(\gamma_1 ,\gamma_2 ,\gamma_3)\,.
\end{aligned}
\end{gather}
Multiplying the left and right hand sides of the equalities \eqref{eqn156}--\eqref{eqn158} by the function $\delta (\gamma_1, \gamma_2, \gamma_3)$ from formula \eqref{eqn201aa} and using the Fourier inversion formula \eqref{A3} we obtain
\begin{align}
&\frac{1}{(2\pi)^{3/2}} \int_{-\infty}^{\infty}\int_{-\infty}^{\infty} \int_{-\infty}^{\infty}U_1(\gamma_1, \gamma_2, \gamma_3, t)\delta (\gamma_1,\gamma_2,\gamma_3)e ^{-i(x_1\gamma_1+x_2\gamma_2+x_3\gamma_3)}\,d\gamma_1d\gamma_2d\gamma_3\nonumber \\
&=\frac{1}{(2\pi)^{3/2}} \int_{-\infty}^{\infty}
\int_{-\infty}^{\infty} \int_{-\infty}^{\infty}
\Big[ \int_0^t e ^{-\nu (\gamma_1^2 +\gamma_2^2
+\gamma_3^2) (t-\tau)} \frac{( \gamma_2^2
+\gamma_3^2)}
 { (\gamma_1^2 +\gamma_2^2+\gamma_3^2 ) } \,\tilde{F}_1( \gamma_1, \gamma_2, \gamma_3, \tau)]d\tau \nonumber\\
&\quad - \int_0^t e ^{-\nu (\gamma_1^2 +\gamma_2^2
 +\gamma_3^2) (t-\tau)}\frac{ [\gamma_1\gamma_2 \tilde{F}_2
( \gamma_1, \gamma_2, \gamma_3, \tau )
+ \gamma_1\gamma_3 \tilde{F}_3( \gamma_1, \gamma_2, \gamma_3, \tau )]}
{ (\gamma_1^2 +\gamma_2^2+\gamma_3^2 ) }\,d\tau \nonumber\\
&\quad  + e ^{-\nu (\gamma_1^2 +\gamma_2^2 +\gamma_3^2) t}
U_1^0(\gamma_1 ,\gamma_2 ,\gamma_3)\Big]\delta (\gamma_1,\gamma_2,\gamma_3)
e ^{-i(x_1\gamma_1+x_2\gamma_2+x_3\gamma_3)}\,d\gamma_1d\gamma_2d\gamma_3
\nonumber \\
\label{eqn160a}
\end{align}

\begin{align}
%&U_2^{*}(x_1, x_2, x_3, t)\nonumber\\ 
&\frac{1}{(2\pi)^{3/2}} \int_{-\infty}^{\infty}\int_{-\infty}^{\infty} \int_{-\infty}^{\infty}U_2(\gamma_1, \gamma_2, \gamma_3, t)\delta (\gamma_1,\gamma_2,\gamma_3)e ^{-i(x_1\gamma_1+x_2\gamma_2+x_3\gamma_3)}\,d\gamma_1d\gamma_2d\gamma_3\nonumber \\
&= \frac{1}{(2\pi)^{3/2}} \int_{-\infty}^{\infty}
 \int_{-\infty}^{\infty} \int_{-\infty}^{\infty}
\Big[ \int_0^t e ^{-\nu (\gamma_1^2 +\gamma_2^2 +\gamma_3^2) (t-\tau)}
\frac{[( \gamma_3^2 +\gamma_1^2)  \tilde{F}_2( \gamma_1, \gamma_2, \gamma_3,
\tau)]} { (\gamma_1^2 +\gamma_2^2+\gamma_3^2 ) } \,d\tau \nonumber\\
&\quad - \int_0^t e ^{-\nu (\gamma_1^2 +\gamma_2^2
 +\gamma_3^2) (t-\tau)}\frac{[\gamma_2\gamma_3 \tilde{F}_3
 ( \gamma_1, \gamma_2, \gamma_3, \tau ) +\gamma_2\gamma_1
\tilde{F}_1( \gamma_1, \gamma_2, \gamma_3, \tau )]}
{ (\gamma_1^2 +\gamma_2^2+\gamma_3^2 ) }\,d\tau \nonumber \\
&\quad +e ^{-\nu (\gamma_1^2 +\gamma_2^2 +\gamma_3^2) t}
U_2^0(\gamma_1 ,\gamma_2 ,\gamma_3)\Big]\delta (\gamma_1,\gamma_2,\gamma_3)
e ^{-i(x_1\gamma_1+x_2\gamma_2+x_3\gamma_3)}\,d\gamma_1d\gamma_2d\gamma_3 \nonumber
\nonumber \\
\label{eqn161a}
\end{align}

\begin{align}
%&U_3^{*}(x_1, x_2, x_3, t)\nonumber\\ 
&\frac{1}{(2\pi)^{3/2}} \int_{-\infty}^{\infty}\int_{-\infty}^{\infty} \int_{-\infty}^{\infty}U_3(\gamma_1, \gamma_2, \gamma_3, t)\delta (\gamma_1,\gamma_2,\gamma_3)e ^{-i(x_1\gamma_1+x_2\gamma_2+x_3\gamma_3)}\,d\gamma_1d\gamma_2d\gamma_3\nonumber \\
&= \frac{1}{(2\pi)^{3/2}} \int_{-\infty}^{\infty}
\int_{-\infty}^{\infty} \int_{-\infty}^{\infty}
\Big[ \int_0^t e ^{-\nu (\gamma_1^2 +\gamma_2^2 +\gamma_3^2)
(t-\tau)} \frac{[( \gamma_1^2 +\gamma_2^2)  \tilde{F}_3( \gamma_1,
\gamma_2, \gamma_3, \tau)]} { (\gamma_1^2 +\gamma_2^2+\gamma_3^2 ) } \,d\tau
\nonumber \\
&\quad - \int_0^t e ^{-\nu (\gamma_1^2 +\gamma_2^2 +\gamma_3^2) (t-\tau)}
\frac{[\gamma_3\gamma_1 \tilde{F}_1( \gamma_1, \gamma_2, \gamma_3, \tau )
+\gamma_3\gamma_2 \tilde{F}_2( \gamma_1, \gamma_2, \gamma_3, \tau )]}
 { (\gamma_1^2 +\gamma_2^2+\gamma_3^2 ) }\,d\tau \nonumber \\
&\quad + e ^{-\nu (\gamma_1^2 +\gamma_2^2 +\gamma_3^2) t}  U_3^0
(\gamma_1 ,\gamma_2 ,\gamma_3)\Big]\delta (\gamma_1,\gamma_2,\gamma_3) e ^{-i(x_1\gamma_1+x_2\gamma_2+x_3\gamma_3)}
\,d\gamma_1d\gamma_2d\gamma_3 \nonumber
\nonumber \\
\label{eqn162a}
\end{align}

\begin{remark} \rm
Right hand sides of the equations \eqref{eqn160a}--\eqref{eqn162a}  have integrands that
contain multipliers

1) fractions $\chi_{ij}(\gamma_1, \gamma_2, \gamma_3)$ with simple features at 
$\gamma_1 = \gamma_2 = \gamma_3 = 0$

\begin{gather*}
\frac{( \gamma_2^2 +\gamma_3^2)} { (\gamma_1^2 +\gamma_2^2+\gamma_3^2 ) },\quad
\frac{( \gamma_1\cdot\gamma_2)} { (\gamma_1^2 +\gamma_2^2+\gamma_3^2 ) },\quad
\frac{( \gamma_1\cdot\gamma_3)} { (\gamma_1^2 +\gamma_2^2+\gamma_3^2 ) }, \\
\frac{( \gamma_2\cdot\gamma_1)} { (\gamma_1^2 +\gamma_2^2+\gamma_3^2 ) },\quad
\frac{( \gamma_3^2 +\gamma_1^2)} { (\gamma_1^2 +\gamma_2^2+\gamma_3^2 ) },\quad
\frac{( \gamma_2\cdot\gamma_3)} { (\gamma_1^2 +\gamma_2^2+\gamma_3^2 ) },\\
\frac{( \gamma_3\cdot\gamma_1)} { (\gamma_1^2 +\gamma_2^2+\gamma_3^2 ) }, \quad
\frac{( \gamma_3\cdot\gamma_2)} { (\gamma_1^2 +\gamma_2^2+\gamma_3^2 ) },\quad
\frac{( \gamma_1^2 +\gamma_2^2)} { (\gamma_1^2 +\gamma_2^2+\gamma_3^2 ) }\\
\\
i,j = 1,2,3.
\end{gather*}

and

2) function $\delta(\gamma_1, \gamma_2, \gamma_3)$ is determined by formula \eqref{eqn201aa} with property \eqref{eqn201zz}.

Consequently integrands belong to space S.
\end{remark}

Further we put (1 - 1 + $\delta(\gamma_1,\gamma_2, \gamma_3)$) instead of $\delta(\gamma_1, \gamma_2, \gamma_3)$ in left hand sides of the equations \eqref{eqn160a}--\eqref{eqn162a}. Then we move integrals with (-1 + $\delta(\gamma_1,\gamma_2,\gamma_3)$) from left hand sides to right hand sides of the equations \eqref{eqn160a}--\eqref{eqn162a}. And we have

\begin{align}
&u_1(x_1, x_2, x_3, t) \nonumber \\
&=\frac{1}{(2\pi)^{3/2}} \int_{-\infty}^{\infty}
\int_{-\infty}^{\infty} \int_{-\infty}^{\infty}
\Big[ \int_0^t e ^{-\nu (\gamma_1^2 +\gamma_2^2
+\gamma_3^2) (t-\tau)} \frac{( \gamma_2^2
+\gamma_3^2)}
 { (\gamma_1^2 +\gamma_2^2+\gamma_3^2 ) } \,\tilde{F}_1( \gamma_1, \gamma_2, \gamma_3, \tau)]d\tau \nonumber\\
&\quad - \int_0^t e ^{-\nu (\gamma_1^2 +\gamma_2^2
 +\gamma_3^2) (t-\tau)}\frac{ [\gamma_1\gamma_2 \tilde{F}_2
( \gamma_1, \gamma_2, \gamma_3, \tau )
+ \gamma_1\gamma_3 \tilde{F}_3( \gamma_1, \gamma_2, \gamma_3, \tau )]}
{ (\gamma_1^2 +\gamma_2^2+\gamma_3^2 ) }\,d\tau \nonumber\\
&\quad  + e ^{-\nu (\gamma_1^2 +\gamma_2^2 +\gamma_3^2) t}
U_1^0(\gamma_1 ,\gamma_2 ,\gamma_3)\Big]\delta (\gamma_1,\gamma_2,\gamma_3)
e ^{-i(x_1\gamma_1+x_2\gamma_2+x_3\gamma_3)}\,d\gamma_1d\gamma_2d\gamma_3
\nonumber \\
&+ \frac{1}{(2\pi)^{3/2}} \int_{-\infty}^{\infty}\int_{-\infty}^{\infty} \int_{-\infty}^{\infty}U_1(\gamma_1, \gamma_2, \gamma_3, t)(1 - \delta (\gamma_1,\gamma_2,\gamma_3))e ^{-i(x_1\gamma_1+x_2\gamma_2+x_3\gamma_3)}\,d\gamma_1d\gamma_2d\gamma_3\nonumber \\
&=\frac{1}{8\pi^3} \int_0^t\int_{-\infty}^{\infty}
\int_{-\infty}^{\infty} \int_{-\infty}^{\infty}
\Big[\frac{( \gamma_2^2 +\gamma_3^2)} { (\gamma_1^2 +\gamma_2^2+\gamma_3^2 ) }
   e ^{-\nu (\gamma_1^2 +\gamma_2^2 +\gamma_3^2) (t-\tau)} \nonumber \\
&\quad\times \int_{-\infty}^{\infty}\int_{-\infty}^{\infty}
\int_{-\infty}^{\infty}e  ^{i(\tilde x_1\gamma_1+\tilde x_2\gamma_2+\tilde x_3
\gamma_3)}  \tilde{f}_1(\tilde x_1,\tilde x_2, \tilde x_3,\tau)\,d\tilde x_1
d\tilde x_2 d\tilde x_3\Big]\delta (\gamma_1,\gamma_2,\gamma_3) \nonumber \\
&\quad\times e ^{-i(x_1\gamma_1+x_2\gamma_2+x_3\gamma_3)}\,d\gamma_1d\gamma_2
d\gamma_3d\tau
 \nonumber \\
&\quad - \frac{1}{8\pi^3} \int_0^t\int_{-\infty}^{\infty}
\int_{-\infty}^{\infty} \int_{-\infty}^{\infty}
\Big[\frac{ \gamma_1\gamma_2} { (\gamma_1^2 +\gamma_2^2+\gamma_3^2 ) }
   e ^{-\nu
(\gamma_1^2 +\gamma_2^2 +\gamma_3^2) (t-\tau)} \nonumber\\
&\quad\times \int_{-\infty}^{\infty}\int_{-\infty}^{\infty}\int_{-\infty}^{\infty}
e^{i(\tilde x_1\gamma_1+\tilde x_2\gamma_2+\tilde x_3\gamma_3)}
\tilde{f}_2(\tilde x_1,\tilde x_2, \tilde x_3,\tau)\,d\tilde x_1
d\tilde x_2 d\tilde x_3\Big] \delta (\gamma_1,\gamma_2,\gamma_3)\nonumber \\
&\quad\times   e ^{-i(x_1\gamma_1+x_2\gamma_2+x_3\gamma_3)}\,d\gamma_1d\gamma_2
d\gamma_3d\tau \nonumber
\\
&\quad- \frac{1}{8\pi^3} \int_0^t\int_{-\infty}^{\infty} \int_{-\infty}^{\infty}
\int_{-\infty}^{\infty}\Big[\frac{ \gamma_1\gamma_3}
{ (\gamma_1^2 +\gamma_2^2+\gamma_3^2 ) }
  e ^{-\nu (\gamma_1^2 +\gamma_2^2
 +\gamma_3^2) (t-\tau)} \nonumber \\
 &\quad\times \int_{-\infty}^{\infty}\int_{-\infty}^{\infty}
\int_{-\infty}^{\infty}e  ^{i(\tilde x_1\gamma_1
+\tilde x_2\gamma_2+\tilde x_3\gamma_3)}
 \tilde{f}_3(\tilde x_1,\tilde x_2, \tilde x_3,\tau)\,d\tilde x_1
d\tilde x_2 d\tilde x_3\Big]\delta (\gamma_1,\gamma_2,\gamma_3)\nonumber \\
&\quad\times e ^{-i(x_1\gamma_1+x_2\gamma_2+x_3\gamma_3)}
\,d\gamma_1d\gamma_2d\gamma_3d\tau \nonumber
 \\
&\quad + \frac{1}{8\pi^3} \int_{-\infty}^{\infty} \int_{-\infty}^{\infty}
\int_{-\infty}^{\infty} e ^{-\nu (\gamma_1^2 +\gamma_2^2 +\gamma_3^2) t}\nonumber  \\
&\quad\times \Big[  \int_{-\infty}^{\infty}\int_{-\infty}^{\infty}
\int_{-\infty}^{\infty}e  ^{i(\tilde x_1\gamma_1+\tilde x_2\gamma_2
+\tilde x_3\gamma_3)}
 u_1^0(\tilde x_1,\tilde x_2, \tilde x_3)\,d\tilde x_1d\tilde x_2 d\tilde x_3\Big]\delta (\gamma_1,\gamma_2,\gamma_3)
\nonumber \\
&\quad\times
e ^{-i(x_1\gamma_1+x_2\gamma_2+x_3\gamma_3)}\,d\gamma_1d\gamma_2d\gamma_3
\nonumber \\
&\quad + \frac{1}{8\pi^3} \int_{-\infty}^{\infty} \int_{-\infty}^{\infty}
\int_{-\infty}^{\infty} \nonumber  \\
&\quad\times \Big[  \int_{-\infty}^{\infty}\int_{-\infty}^{\infty}
\int_{-\infty}^{\infty}e  ^{i(\tilde x_1\gamma_1+\tilde x_2\gamma_2
+\tilde x_3\gamma_3)}
 u_1(\tilde x_1,\tilde x_2, \tilde x_3, t)\,d\tilde x_1d\tilde x_2 d\tilde x_3\Big](1 - \delta (\gamma_1,\gamma_2,\gamma_3))
\nonumber \\
&\quad\times
e ^{-i(x_1\gamma_1+x_2\gamma_2+x_3\gamma_3)}\,d\gamma_1d\gamma_2d\gamma_3
\nonumber \\
&= S_{11}(\tilde{f}_1)+ S_{12}(\tilde{f}_2)+ S_{13}(\tilde{f}_3)+B(u_1^0)+E (u_1),
\label{eqn160}
\end{align}
 \\
 \\ 
 \\  
\begin{align}
&u_2(x_1, x_2, x_3, t) \nonumber  \\
&= \frac{1}{(2\pi)^{3/2}} \int_{-\infty}^{\infty}
 \int_{-\infty}^{\infty} \int_{-\infty}^{\infty}
\Big[ \int_0^t e ^{-\nu (\gamma_1^2 +\gamma_2^2 +\gamma_3^2) (t-\tau)}
\frac{[( \gamma_3^2 +\gamma_1^2)  \tilde{F}_2( \gamma_1, \gamma_2, \gamma_3,
\tau)]} { (\gamma_1^2 +\gamma_2^2+\gamma_3^2 ) } \,d\tau \nonumber\\
&\quad - \int_0^t e ^{-\nu (\gamma_1^2 +\gamma_2^2
 +\gamma_3^2) (t-\tau)}\frac{[\gamma_2\gamma_3 \tilde{F}_3
 ( \gamma_1, \gamma_2, \gamma_3, \tau ) +\gamma_2\gamma_1
\tilde{F}_1( \gamma_1, \gamma_2, \gamma_3, \tau )]}
{ (\gamma_1^2 +\gamma_2^2+\gamma_3^2 ) }\,d\tau \nonumber \\
&\quad +e ^{-\nu (\gamma_1^2 +\gamma_2^2 +\gamma_3^2) t}
U_2^0(\gamma_1 ,\gamma_2 ,\gamma_3)\Big]\delta (\gamma_1,\gamma_2,\gamma_3)
e ^{-i(x_1\gamma_1+x_2\gamma_2+x_3\gamma_3)}\,d\gamma_1d\gamma_2d\gamma_3 \nonumber
 \\
&+ \frac{1}{(2\pi)^{3/2}} \int_{-\infty}^{\infty}\int_{-\infty}^{\infty} \int_{-\infty}^{\infty}U_2(\gamma_1, \gamma_2, \gamma_3, t)(1 - \delta (\gamma_1,\gamma_2,\gamma_3))e ^{-i(x_1\gamma_1+x_2\gamma_2+x_3\gamma_3)}\,d\gamma_1d\gamma_2d\gamma_3\nonumber \\
&= -\frac{1}{8\pi^3} \int_0^t\int_{-\infty}^{\infty}
 \int_{-\infty}^{\infty} \int_{-\infty}^{\infty}\Big[\frac{ \gamma_2\gamma_1}
{ (\gamma_1^2 +\gamma_2^2+\gamma_3^2 ) }   
e ^{-\nu (\gamma_1^2 +\gamma_2^2 +\gamma_3^2) (t-\tau)}  \nonumber \\
&\quad\times \int_{-\infty}^{\infty}\int_{-\infty}^{\infty}
\int_{-\infty}^{\infty}e  ^{i(\tilde x_1\gamma_1+\tilde x_2\gamma_2
+\tilde x_3\gamma_3)}
 \tilde{f}_1(\tilde x_1,\tilde x_2, \tilde x_3,\tau)\,d\tilde x_1
d\tilde x_2 d\tilde x_3\Big] \delta (\gamma_1,\gamma_2,\gamma_3)\nonumber \\
&\quad\times e ^{-i(x_1\gamma_1+x_2\gamma_2+x_3\gamma_3)}\,d\gamma_1
d\gamma_2d\gamma_3d\tau \nonumber
\\
&\quad + \frac{1}{8\pi^3} \int_0^t\int_{-\infty}^{\infty}
 \int_{-\infty}^{\infty} \int_{-\infty}^{\infty}
\Big[\frac{ (\gamma_3^2 + \gamma_1^2)} { (\gamma_1^2 +\gamma_2^2+\gamma_3^2 ) }
   e ^{-\nu (\gamma_1^2 +\gamma_2^2 +\gamma_3^2)
(t-\tau)} \nonumber \\
&\quad\times  \int_{-\infty}^{\infty}\int_{-\infty}^{\infty}
\int_{-\infty}^{\infty}e  ^{i(\tilde x_1\gamma_1+\tilde x_2\gamma_2
+\tilde x_3\gamma_3)}
\tilde{f}_2(\tilde x_1,\tilde x_2, \tilde x_3,\tau)\,d\tilde x_1d\tilde x_2
 d\tilde x_3\Big]\delta (\gamma_1,\gamma_2,\gamma_3) \nonumber\\
&\quad\times  e ^{-i(x_1\gamma_1+x_2\gamma_2+x_3\gamma_3)}
\,d\gamma_1d\gamma_2d\gamma_3d\tau \nonumber
\\
&\quad - \frac{1}{8\pi^3} \int_0^t\int_{-\infty}^{\infty}
\int_{-\infty}^{\infty} \int_{-\infty}^{\infty}
\Big[\frac{ \gamma_2\gamma_3} { (\gamma_1^2 +\gamma_2^2+\gamma_3^2 ) }
   e ^{-\nu (\gamma_1^2 +\gamma_2^2 +\gamma_3^2)
 (t-\tau)} \nonumber\\
&\quad\times \int_{-\infty}^{\infty}\int_{-\infty}^{\infty}
\int_{-\infty}^{\infty}e  ^{i(\tilde x_1\gamma_1+\tilde x_2\gamma_2
+\tilde x_3\gamma_3)}
 \tilde{f}_3(\tilde x_1,\tilde x_2, \tilde x_3,\tau)\,d\tilde x_1
d\tilde x_2 d\tilde x_3\Big] \delta (\gamma_1,\gamma_2,\gamma_3)\nonumber\\
 &\quad \times e ^{-i(x_1\gamma_1+x_2\gamma_2+x_3\gamma_3)}\,d\gamma_1
d\gamma_2d\gamma_3d\tau
\nonumber  \\
&\quad + \frac{1}{8\pi^3} \int_{-\infty}^{\infty}
\int_{-\infty}^{\infty} \int_{-\infty}^{\infty}
e ^{-\nu (\gamma_1^2 +\gamma_2^2 +\gamma_3^2) t}\nonumber \\
&\quad\times \Big[  \int_{-\infty}^{\infty}\int_{-\infty}^{\infty}
\int_{-\infty}^{\infty}e  ^{i(\tilde x_1\gamma_1+\tilde x_2\gamma_2
+\tilde x_3\gamma_3)}
 u_2^0(\tilde x_1,\tilde x_2, \tilde x_3)\,d\tilde x_1d\tilde x_2
d\tilde x_3\Big] \delta (\gamma_1,\gamma_2,\gamma_3)\nonumber \\
&\quad\times e ^{-i(x_1\gamma_1+x_2\gamma_2+x_3\gamma_3)}\,d\gamma_1
d\gamma_2d\gamma_3 \nonumber \\
&\quad + \frac{1}{8\pi^3} \int_{-\infty}^{\infty}
\int_{-\infty}^{\infty} \int_{-\infty}^{\infty}
\nonumber \\
&\quad\times \Big[  \int_{-\infty}^{\infty}\int_{-\infty}^{\infty}
\int_{-\infty}^{\infty}e  ^{i(\tilde x_1\gamma_1+\tilde x_2\gamma_2
+\tilde x_3\gamma_3)}
 u_2(\tilde x_1,\tilde x_2, \tilde x_3, t)\,d\tilde x_1d\tilde x_2
d\tilde x_3\Big] (1 - \delta (\gamma_1,\gamma_2,\gamma_3))\nonumber \\
&\quad\times e ^{-i(x_1\gamma_1+x_2\gamma_2+x_3\gamma_3)}\,d\gamma_1
d\gamma_2d\gamma_3 \nonumber \\
&= S_{21}(\tilde{f}_1)+ S_{22}(\tilde{f}_2)+ S_{23}(\tilde{f}_3)+B(u_2^0)+E(u_2),
\label{eqn161}
\end{align}
 \\
 \\ 
 \\  
\begin{align}
&u_3(x_1, x_2, x_3, t) \nonumber \\
&= \frac{1}{(2\pi)^{3/2}} \int_{-\infty}^{\infty}
\int_{-\infty}^{\infty} \int_{-\infty}^{\infty}
\Big[ \int_0^t e ^{-\nu (\gamma_1^2 +\gamma_2^2 +\gamma_3^2)
(t-\tau)} \frac{[( \gamma_1^2 +\gamma_2^2)  \tilde{F}_3( \gamma_1,
\gamma_2, \gamma_3, \tau)]} { (\gamma_1^2 +\gamma_2^2+\gamma_3^2 ) } \,d\tau
\nonumber \\
&\quad - \int_0^t e ^{-\nu (\gamma_1^2 +\gamma_2^2 +\gamma_3^2) (t-\tau)}
\frac{[\gamma_3\gamma_1 \tilde{F}_1( \gamma_1, \gamma_2, \gamma_3, \tau )
+\gamma_3\gamma_2 \tilde{F}_2( \gamma_1, \gamma_2, \gamma_3, \tau )]}
 { (\gamma_1^2 +\gamma_2^2+\gamma_3^2 ) }\,d\tau \nonumber \\
&\quad + e ^{-\nu (\gamma_1^2 +\gamma_2^2 +\gamma_3^2) t}  U_3^0
(\gamma_1 ,\gamma_2 ,\gamma_3)\Big]\delta (\gamma_1,\gamma_2,\gamma_3) e ^{-i(x_1\gamma_1+x_2\gamma_2+x_3\gamma_3)}
\,d\gamma_1d\gamma_2d\gamma_3 \nonumber
\\
&+ \frac{1}{(2\pi)^{3/2}} \int_{-\infty}^{\infty}\int_{-\infty}^{\infty} \int_{-\infty}^{\infty}U_3(\gamma_1, \gamma_2, \gamma_3, t)(1 - \delta (\gamma_1,\gamma_2,\gamma_3))e ^{-i(x_1\gamma_1+x_2\gamma_2+x_3\gamma_3)}\,d\gamma_1d\gamma_2d\gamma_3\nonumber \\
&=-\frac{1}{8\pi^3} \int_0^t\int_{-\infty}^{\infty}
\int_{-\infty}^{\infty} \int_{-\infty}^{\infty}
\Big[\frac{ \gamma_3\gamma_1} { (\gamma_1^2 +\gamma_2^2+\gamma_3^2 ) }
   e ^{-\nu (\gamma_1^2 +\gamma_2^2 +\gamma_3^2) (t-\tau)}
\nonumber \\
&\quad\times \int_{-\infty}^{\infty}\int_{-\infty}^{\infty}
\int_{-\infty}^{\infty}e  ^{i(\tilde x_1\gamma_1+\tilde x_2\gamma_2
+\tilde x_3\gamma_3)} \tilde{f}_1(\tilde x_1,\tilde x_2, \tilde x_3,\tau)
\,d\tilde x_1d\tilde x_2 d\tilde x_3\Big] \delta (\gamma_1,\gamma_2,\gamma_3)\nonumber \\
&\quad\times e ^{-i(x_1\gamma_1+x_2\gamma_2+x_3\gamma_3)}\,d\gamma_1d\gamma_2
d\gamma_3d\tau \nonumber \\
&\quad -\frac{1}{8\pi^3} \int_0^t\int_{-\infty}^{\infty}
\int_{-\infty}^{\infty} \int_{-\infty}^{\infty}
\Big[\frac{ \gamma_3\gamma_2} { (\gamma_1^2 +\gamma_2^2+\gamma_3^2 ) }
   e ^{-\nu (\gamma_1^2 +\gamma_2^2 +\gamma_3^2)
(t-\tau)}  \nonumber \\
&\quad\times \int_{-\infty}^{\infty}\int_{-\infty}^{\infty}
\int_{-\infty}^{\infty}e  ^{i(\tilde x_1\gamma_1+\tilde x_2\gamma_2
+\tilde x_3\gamma_3)}
 \tilde{f}_2(\tilde x_1,\tilde x_2, \tilde x_3,\tau)\,d\tilde x_1
d\tilde x_2 d\tilde x_3\Big] \delta (\gamma_1,\gamma_2,\gamma_3)\nonumber \\
&\quad \times e ^{-i(x_1\gamma_1+x_2\gamma_2+x_3\gamma_3)}
\,d\gamma_1d\gamma_2d\gamma_3d\tau
\nonumber \\
&\quad +\frac{1}{8\pi^3} \int_0^t\int_{-\infty}^{\infty}
\int_{-\infty}^{\infty} \int_{-\infty}^{\infty}
\Big[\frac{ (\gamma_1^2 + \gamma_2^2)} { (\gamma_1^2 +\gamma_2^2+\gamma_3^2 ) }
  e ^{-\nu (\gamma_1^2 +\gamma_2^2 +\gamma_3^2) (t-\tau)}
\nonumber \\
&\quad\times \int_{-\infty}^{\infty}\int_{-\infty}^{\infty}\int_{-\infty}^{\infty}
e  ^{i(\tilde x_1\gamma_1+\tilde x_2\gamma_2+\tilde x_3\gamma_3)}
 \tilde{f}_3(\tilde x_1,\tilde x_2, \tilde x_3,\tau)\,d\tilde x_1
d\tilde x_2 d\tilde x_3\Big] \delta (\gamma_1,\gamma_2,\gamma_3)\nonumber \\
&\quad \times e ^{-i(x_1\gamma_1+x_2\gamma_2+x_3\gamma_3)}\,d\gamma_1
d\gamma_2d\gamma_3d\tau \nonumber \\
&\quad + \frac{1}{8\pi^3} \int_{-\infty}^{\infty}
\int_{-\infty}^{\infty} \int_{-\infty}^{\infty}
e ^{-\nu (\gamma_1^2 +\gamma_2^2 +\gamma_3^2) t} \nonumber \\
&\quad\times \Big[  \int_{-\infty}^{\infty}
 \int_{-\infty}^{\infty}\int_{-\infty}^{\infty}
 e  ^{i(\tilde x_1\gamma_1+\tilde x_2\gamma_2+\tilde x_3\gamma_3)}
 u_3^0(\tilde x_1,\tilde x_2, \tilde x_3)\,d\tilde x_1d\tilde x_2
d\tilde x_3\Big] \delta (\gamma_1,\gamma_2,\gamma_3)\nonumber \\
&\quad\times e ^{-i(x_1\gamma_1+x_2\gamma_2+x_3\gamma_3)}
\,d\gamma_1d\gamma_2d\gamma_3 \nonumber \\
&\quad + \frac{1}{8\pi^3} \int_{-\infty}^{\infty}
\int_{-\infty}^{\infty} \int_{-\infty}^{\infty}
\nonumber \\
&\quad\times \Big[  \int_{-\infty}^{\infty}\int_{-\infty}^{\infty}
\int_{-\infty}^{\infty}e  ^{i(\tilde x_1\gamma_1+\tilde x_2\gamma_2
+\tilde x_3\gamma_3)}
 u_3(\tilde x_1,\tilde x_2, \tilde x_3, t)\,d\tilde x_1d\tilde x_2
d\tilde x_3\Big] (1 - \delta (\gamma_1,\gamma_2,\gamma_3))\nonumber \\
&\quad\times e ^{-i(x_1\gamma_1+x_2\gamma_2+x_3\gamma_3)}\,d\gamma_1
d\gamma_2d\gamma_3 \nonumber \\
&= S_{31}(\tilde{f}_1)+ S_{32}(\tilde{f}_2)+ S_{33}(\tilde{f}_3)+B(u_3^0)+E(u_3).
\label{eqn162}
\end{align}

Here $S_{11}()$, $S_{12}()$, $S_{13}()$, $S_{21}()$, $S_{22}()$,
$S_{23}()$, $S_{31}()$, $S_{32}()$, $S_{33}()$, $B()$, $E()$ are integral
operators, and satisfy
\[
S_{12}()=S_{21}(),\quad S_{13}()=S_{31}() ,\quad S_{23}()=S_{32}()\,.
 \]

\begin{remark} \rm
It should be noted that for t = 0 multiple integrals, containing integral $\int_0^t$, equal zero and the formulas \eqref{eqn160}--\eqref{eqn162} easily converted to the form

$u_i(x_1, x_2, x_3, 0) = u_i^0(x_1, x_2, x_3), i = 1,2,3.$
\end{remark}

From the three expressions above for $u_1,u_2,u_3$ \eqref{eqn160}--\eqref{eqn162}, it
follows that the vector $\vec{u}$ can be represented as:
\begin{equation}\label{eqn164}
\vec{u}=\bar{\bar{S}}\cdot\vec{\tilde{f}}+\bar{\bar{B}}\cdot\vec{u}^0+\bar{\bar{E}}\cdot\vec{u}
=\bar{\bar{S}}\cdot\vec{f}- \bar{\bar{S}}\cdot(\vec{u}\cdot\nabla)\vec{u}
+\bar{\bar{B}}\cdot\vec{u}^0+\bar{\bar{E}}\cdot\vec{u}\,,
\end{equation}
where $\vec{\tilde{f}}$ is determined by formula \eqref{eqn14}.

Here $\bar{\bar{S}}$, $\bar{\bar{B}}$ and $\bar{\bar{E}}$ are the matrix integral operators:
\begin{equation}\label{eqn164m}
 \begin{pmatrix}
S_{11} & S_{12} & S_{13} \\
S_{21} & S_{22} & S_{23} \\
S_{31} & S_{32} & S_{33}  \end{pmatrix}, \quad
 \begin{pmatrix}
B & 0 & 0 \\
0 & B & 0 \\
0 & 0 & B  \end{pmatrix}, \quad
 \begin{pmatrix}
E & 0 & 0 \\
0 & E & 0 \\
0 & 0 & E  \end{pmatrix}.
\end{equation}

\section{Equivalence of the Cauchy problem in differential form
 \eqref{eqn1}--\eqref{eqn3} and in integral form}

Let us denote solution of  \eqref{eqn1}--\eqref{eqn3} as
\{$\vec{u}(x_1, x_2, x_3, t)$, p($x_1, x_2, x_3,$ t)\};
in other words let us consider the infinitely differentiable
by $t \in [0,\infty)$  vector-function
$\vec{u}(x_1, x_2, x_3, t) \in \overrightarrow{TS}$,
and infinitely differentiable
function $p(x_1, x_2, x_3, t) \in S$,
that turn equations \eqref{eqn1} and \eqref{eqn2} into identities.
Vector-function $\vec{u}(x_1, x_2, x_3, t)$ also satisfies the initial
condition \eqref{eqn3} $(\vec{u}^0 (x_1, x_2, x_3)\in \overrightarrow{TS})$:
\begin{equation}\label{eqn202}
\vec{u}(x_1, x_2, x_3, t)|_{t = 0}=\vec{u}^0(x_1, x_2, x_3)
\end{equation}

Let us put $\{\vec{u}(x_1, x_2, x_3, t), p(x_1, x_2, x_3, t)\}$ into
equations \eqref{eqn1}, \eqref{eqn2} and apply Fourier and Laplace transforms
to the resulting identities considering initial condition \eqref{eqn3}. After
all required operations (as in sections 2 and 4) we receive that vector-function
$\vec{u}(x_1, x_2, x_3, t)$ satisfies integral equation
\begin{equation}\label{eqn203}
\vec{u}=\bar{\bar{S}}\cdot \vec{f}-
\bar{\bar{S}}\cdot(\vec{u}\cdot\nabla)\vec{u}
+\bar{\bar{B}}\cdot\vec{u}^0+\bar{\bar{E}}\cdot\vec{u}\; = \bar{\bar{S}}^{\nabla}\cdot\vec{u}
\end{equation}
Then the vector-function grad $p\in \overrightarrow{TS}$ is defined by
equations \eqref{eqn1} where vector-function $\vec{u}$ is defined
by \eqref{eqn203}.

Here $\vec{f} \in \overrightarrow{TS}$, $\vec{u}^0 \in \overrightarrow{TS}$
and $\bar{\bar{S}},\bar{\bar{B}},\bar{\bar{E}},\bar{\bar{S}}^{\nabla}$ are matrix integral
operators.
Vector-functions
$ \bar{\bar{S}}\cdot\vec{f}$, $\bar{\bar{B}}\cdot\vec{u}^0$, $\bar{\bar{E}}\cdot\vec{u}$,
$\bar{\bar{S}}\cdot(\vec{u}\cdot\nabla)\vec{u}$  also belong
$\overrightarrow{TS}$ since the Fourier transform maps the Space $\overrightarrow{TS}$ onto the Space $\overrightarrow{TS}$, and
vice versa the inverse Fourier transform maps the Space $\overrightarrow{TS}$ onto the Space $\overrightarrow{TS}$.

Going from the other side, let us assume that
$\vec{u}(x_1, x_2, x_3, t)\in \overrightarrow{TS}$ is continuous for
$t \in [0,\infty)$ solution of integral equation \eqref{eqn203}.
Integral-operators $S_{ij}\cdot(\vec{u}\cdot\nabla)\vec{u}$ [see \eqref{eqn160}--\eqref{eqn162}] are continuous
for $t \in [0,\infty)$. From here we
obtain that according to \eqref{eqn203},
$$
\vec{u}(x_1, x_2, x_3, 0)  = \vec{u}^0(x_1, x_2, x_3)
$$
also that $\vec{u}(x_1, x_2, x_3, t)$ is differentiable by
$t \in [0,\infty)$. As described before, the Fourier transform maps
the Space $\overrightarrow{TS}$ onto the Space $\overrightarrow{TS}$, and
vice versa the inverse Fourier transform maps the Space $\overrightarrow{TS}$ onto the Space $\overrightarrow{TS}$.
Hence, $\{\vec{u}(x_1, x_2, x_3, t)  \text{ and}$
$\;\; p(x_1, x_2, x_3, t)\}$ is the
solution of the Cauchy problem \eqref{eqn1}--\eqref{eqn3}.
From here we see that solving the Cauchy problem \eqref{eqn1}--\eqref{eqn3}
 is equivalent to finding continuous in $t \in [0,\infty)$ solution of
integral equation \eqref{eqn203}.

\section{The properties of the matrix integral operators  $\bar{\bar{B}},\;\bar{\bar{E}},\;\bar{\bar{S}}$.}

Further we have $\vec{f}\equiv 0$. Let us rewrite integral equation
 \eqref{eqn203} with this condition as
\begin{equation}\label{eqn203ab7}
\vec{u}=- \bar{\bar{S}}\cdot(\vec{u}\cdot\nabla)\vec{u}+\bar{\bar{E}}\cdot\vec{u}+\bar{\bar{B}}
\cdot\vec{u}^0
\end{equation}
$\vec{u}^0(x_1,x_2,x_3) \in \overrightarrow{TS}$.

The integral equation \eqref{eqn203ab7} shows that as the Fourier transform maps the Space $\overrightarrow{TS}$ onto the Space $\overrightarrow{TS}$, and
vice versa the inverse Fourier transform maps the Space $\overrightarrow{TS}$ onto the Space $\overrightarrow{TS}$  then the solution of this integral equation \eqref{eqn203ab7} we will  seek as a vector-function of the Space $\overrightarrow{TS}$.

To solve the integral equation \eqref{eqn203ab7} we will use the null norm.
\begin{equation}\label{eqn203t317v}
|\;\; | \equiv ||\;\;  ||_{0} \;\;\;i.e. \;\;p = 0
\end{equation}
In other words we use for functions $\varphi$ the norm from formula \eqref{eqn200} with $p = 0$ and for vector-functions $\vec{\varphi}$ the norm from formula \eqref{eqn201} with $p = 0.$

Let us describe in details the properties of the matrix integral operator $\bar{\bar{B}}$.
\begin{equation}\label{eqn164n7}
\bar{\bar{B}} = 
 \begin{pmatrix}
B & 0 & 0 \\
0 & B & 0 \\
0 & 0 & B  \end{pmatrix} \quad 
\end{equation}
%(see formula $\;\;\eqref{eqn164m}$).
Here the components $B$ of the matrix integral operator $\bar{\bar{B}}$ have the following representation:
\begin{align}
&B(u_{{i}}^0) = \frac{1}{8\pi^3} \int_{-\infty}^{\infty} \int_{-\infty}^{\infty}
\int_{-\infty}^{\infty} e ^{-\nu (\gamma_{1}^2 +\gamma_{2}^2 +\gamma_{3}^2) t}
\nonumber  \\
&\quad\times \Big[  \int_{-\infty}^{\infty}\int_{-\infty}^{\infty}
\int_{-\infty}^{\infty}e  ^{i(\tilde x_{1}\gamma_{1}+\tilde x_{2}\gamma_{2}
+\tilde x_{3}\gamma_{3})}
 u_{i}^0(\tilde x_{1},\tilde x_{2}, \tilde x_{3})\,d\tilde x_{1}d\tilde x_{2} d\tilde x_{3}\Big]
\nonumber \\
&\quad\times
\delta (\gamma_{1},\gamma_{2},\gamma_{3})e ^{-i(x_{1}\gamma_{1}+x_{2}\gamma_{2}+x_{3}\gamma_{3})}\,d\gamma_{1}d\gamma_{2}d\gamma_{3}
\label{eqn160n17}
\end{align}
\[i = 1, 2, 3.\]
Equation \eqref{eqn160n17} is received from the formulas  \eqref{eqn160}--\eqref{eqn162}   and $u_{i}^0 \in S$, $B(u_{i}^0) \in S$, where S is the Schwartz space. 

To make reading easier we copy the formula$\;\;\eqref{eqn201aa}$  at this place:
\begin{equation}\label{eqn160n2n7}
\delta (\gamma_{1},\gamma_{2},\gamma_{3}) = \text{\LARGE{e}}  ^{-\frac{\displaystyle\epsilon^3}{(\gamma_{1}^2 + \gamma_{2}^2 + \gamma_{3}^2)}}\;,\;\; 0 < \epsilon << 1,\;\;\;\epsilon \neq 0.
\end{equation}

The integral operator $B$ maps the space S into itself. It follows from the basic properties of the Fourier transforms of the functions of the space S (e.g.  \cite{GC68} and/or section 3).

\begin{lemma} \label{lm17}
The integral operator $B$ is a linear operator.
\end{lemma}
\begin{proof} 
1) The Schwartz space $S$ is a linear space (obviously).

$\;\;\;\;\;\;\;2) \; B(\lambda_{1}u_{i1} + \lambda_{2}u_{i2})$ = $\lambda_{1}B(u_{i1}) + \lambda_{2}B(u_{i2})$

for any $u_{i1}, u_{i2} \in S$ and any scalars $\lambda_{1}, \lambda_{2}$.

Proposition 2) follows from the properties of the integration operation.

 Q.E.D.
\end{proof}
\begin{theorem} \label{thm127}
The linear integral operator $B$ is defined everywhere in the space S, with values in the space S. The operator $B$ is bounded for functions of the space S.
\end{theorem}
\begin{proof}
Rewrite operator $B$ from formula  \eqref{eqn160n17} in the form
\begin{equation}\label{eqn160n2n1gh79}
B(u_{i}^0) = F^{-1}[A\cdot F[u_{i}^0]]
\end{equation}
\[i = 1, 2, 3.\]
Here $F$ and $F^{-1}$ are a Fourier transform   and a inverse Fourier transform, respectively, $u_{i}^0 \in S$, $B(u_{i}^0) \in S$.
\begin{equation}\label{eqn160n2n1gh7}
A = e ^{-\nu (\gamma_{1}^2 +\gamma_{2}^2 +\gamma_{3}^2) t}\cdot \delta (\gamma_{1{}},\gamma_{2{}},\gamma_{3{}})
\end{equation}
\[F[u_{i}^0] = \frac{1}{(2\pi)^{3/2}}\int_{-\infty}^{\infty}\int_{-\infty}^{\infty}
\int_{-\infty}^{\infty}e  ^{i(\tilde x_{1{}}\gamma_{1{}}+\tilde x_{2{}}\gamma_{2{}}
+\tilde x_{3{}}\gamma_{3{}})}\times\]
\[\times u_{i}^0(\tilde x_{1{}},\tilde x_{2{}}, \tilde x_{3{}})\,d\tilde x_{1{}}d\tilde x_{2{}} d\tilde x_{3{}}\]
The function $A$ of the formula \eqref{eqn160n2n1gh7} is a continuous, even function of the coordinates $\gamma_{1{}},\gamma_{2{}},\gamma_{3{}}$, that is
\[A(-\gamma_{{}}) = A(\gamma_{}),\;\;\; \vee \; \gamma_{} \in [- \Gamma , \Gamma]\]
and what is more
\begin{equation}\label{eqn160n2n1gh791}
0 \leq A < 1.
\end{equation}
In case if   $A\equiv 1$, from formula \eqref{eqn160n2n1gh79}, we have a known result
\[B_{}(u_{i}^0) \equiv u_{i}^0\]
If $0 < A < 1\;\;$, A - const., it is evident from the formula \eqref{eqn160n2n1gh79} that
\[|B_{}(u_{i}^0)| <  |u_{i}^0|\]
As known the function $F[u_{i}^0]\cdot e ^{-i(x_{1{}}\gamma_{1{}}+x_{2{}}\gamma_{2{}}+x_{3{}}\gamma_{3{}})}$ can be represented as the sum of the even and odd functions
\begin{align}
&F[u_{i}^0]\cdot e ^{-i(x_{1{}}\gamma_{1{}}+x_{2{}}\gamma_{2{}}+x_{3{}}\gamma_{3{}})} = ( F[u_{i}^0]\cdot e ^{-i(x_{1{}}\gamma_{1{}}+x_{2{}}\gamma_{2{}}+x_{3{}}\gamma_{3{}})})_{even} + 
\nonumber \\
&\quad\quad\quad\quad\quad\quad\quad + (F[u_{i}^0]\cdot e ^{-i(x_{1{}}\gamma_{1{}}+x_{2{}}\gamma_{2{}}+x_{3{}}\gamma_{3{}})})_{odd}
\label{eqn160n17g}
\end{align}
Here $( F[u_{i}^0]\cdot e ^{-i(x_{1{}}\gamma_{1{}}+x_{2{}}\gamma_{2{}}+x_{3{}}\gamma_{3{}})})_{even}$ is the even function in all three coordinates , $(F[u_{i}^0]\cdot e ^{-i(x_{1{}}\gamma_{1{}}+x_{2{}}\gamma_{2{}}+x_{3{}}\gamma_{3{}})})_{odd}$  is the sum of seven functions, each of which at least one coordinate is an odd.

As known

1. The product of two even functions is even.

2. The product of the even and odd functions is odd.

3. The integral of an odd function by symmetric within is equal zero. The Fourier transform and inverse Fourier transform are the integrals over symmetrical areas.

From the formula \eqref{eqn160n17g} using the function $A$ we have
\begin{align}
&A F[u_{i}^0]\cdot e ^{-i(x_{1{}}\gamma_{1{}}+x_{2{}}\gamma_{2{}}+x_{3{}}\gamma_{3{}})} = A( F[u_{i}^0]\cdot e ^{-i(x_{1{}}\gamma_{1{}}+x_{2{}}\gamma_{2{}}+x_{3{}}\gamma_{3{}})})_{even} + 
\nonumber \\
&\quad\quad\quad\quad\quad\quad\quad + A(F[u_{i}^0]\cdot e ^{-i(x_{1{}}\gamma_{1{}}+x_{2{}}\gamma_{2{}}+x_{3{}}\gamma_{3{}})})_{odd}
\label{eqn160n17f}
\end{align}
Since A is the even function, we have

$A( F[u_{i}^0]\cdot e ^{-i(x_{1{}}\gamma_{1{}}+x_{2{}}\gamma_{2{}}+x_{3{}}\gamma_{3{}})})_{even}$ is the even function (Rule 1)

$A( F[u_{i}^0]\cdot e ^{-i(x_{1{}}\gamma_{1{}}+x_{2{}}\gamma_{2{}}+x_{3{}}\gamma_{3{}})})_{odd}$ is the odd function (Rule 2)

The inverse Fourier transform $F^{-1}$ is the integral over symmetrical areas.

Then from formula \eqref{eqn160n2n1gh79} in accordance with Rule 3 we have
\begin{align}
&\quad\quad\quad\quad\quad\quad\quad B_{}(u_{i}^0) = F^{-1}[A\cdot F[u_{i}^0]] =
\nonumber \\
&= \int_{-\infty}^{\infty}\int_{-\infty}^{\infty}\int_{-\infty}^{\infty}A\cdot( F[u_{i}^0]\cdot e ^{-i(x_{1{}}\gamma_{1{}}+x_{2{}}\gamma_{2{}}+x_{3{}}\gamma_{3{}})})_{even}\,d\gamma_1
d\gamma_2d\gamma_3 \nonumber \\
\label{eqn160n2n1gh79r}
\end{align}
\[i = 1, 2, 3.\]
In case if   $A\equiv 1$, from formula \eqref{eqn160n2n1gh79r}, we have 
\begin{equation}\label{eqn160n2n1gh7jl1}
B(u_{i}^0) \equiv u_{i}^0
\end{equation}
If $0 < A < 1\;\;$, A - const., it is evident from the formula \eqref{eqn160n2n1gh79r} that
\begin{equation}\label{eqn160n2n1gh7jl2}
|B_{}(u_{i}^0)| <  |u_{i}^0|
\end{equation}
From formula \eqref{eqn160n2n1gh7} we have
\begin{equation}\label{eqn160n2n1gh7jl}
0 \leq A = e ^{-\nu_{} (\gamma_{1{}}^2 +\gamma_{2{}}^2 +\gamma_{3{}}^2) t}\cdot \delta (\gamma_{1{}},\gamma_{2{}},\gamma_{3{}}) < 1.
\end{equation}
The function $A$ of the formula \eqref{eqn160n2n1gh7jl} is a continuous, even function of the coordinates $\gamma_{1{}},\gamma_{2{}},\gamma_{3{}}$.

We use the formulas \eqref{eqn160n2n1gh7jl1}, \eqref{eqn160n2n1gh7jl2} and take the function $A$ from formula \eqref{eqn160n2n1gh7jl}. Then in accordance with the rules of integration we obtain from formula \eqref{eqn160n2n1gh79r}
\begin{align}
& |B_{}(u_{i}^0)| <  |u_{i}^0|
\nonumber \\
\label{eqn160n2n1gh79ee}
\end{align}
\[i = 1, 2, 3.\]
i.e. the operator $B_{}$ is bounded for functions of the space S.

 Q.E.D.
\end{proof}

Thus, the Theorem 6.1 implies that the linear integral operator $B_{}$ is bounded for functions of the space S and therefore the matrix integral operator $\bar{\bar{B}}_{}$ is bounded for vector-functions of the space $\overrightarrow{TS}$ and 
\begin{equation}\label{eqn160n2n17b}
|\bar{\bar{B}}_{}\cdot\vec{u}_{}^0| <  |\vec{u}_{}^0|,  
\end{equation}
where $\vec{u}_{}^0 \in \overrightarrow{TS}$ and $\bar{\bar{B}}_{}\cdot\vec{u}_{}^0 \in \overrightarrow{TS}$.

Let us describe in details the properties of the matrix integral operator $\bar{\bar{E}}_{}$.
\begin{equation}\label{eqn164n27}
\bar{\bar{E}}_{} = 
 \begin{pmatrix}
E_{} & 0 & 0 \\
0 & E_{} & 0 \\
0 & 0 & E_{}  \end{pmatrix} \quad 
\end{equation}
%(see formula $\;\;\eqref{eqn164m}$).
Here the components $E_{}$ of the matrix integral operator $\bar{\bar{E}}_{}$ have the following representation:
\begin{align}
&E_{}(u_{i}) = \frac{1}{8\pi^3} \int_{-\infty}^{\infty} \int_{-\infty}^{\infty}
\int_{-\infty}^{\infty} 
\nonumber  \\
&\quad\times \Big[  \int_{-\infty}^{\infty}\int_{-\infty}^{\infty}
\int_{-\infty}^{\infty}e  ^{i(\tilde x_{1{}}\gamma_{1{}}+\tilde x_{2{}}\gamma_{2{}}+\tilde x_{3{}}\gamma_{3{}})}
 u_{i}(\tilde x_{1{}},\tilde x_{2{}}, \tilde x_{3{}}, t)\,d\tilde x_{1{}}d\tilde x_{2{}} d\tilde x_{3{}}\Big]
\nonumber \\
&\quad\times
(1 - \delta (\gamma_{1{}},\gamma_{2{}},\gamma_{3{}})) e ^{-i(x_{1{}}\gamma_{1{}}+x_{2{}}\gamma_{2{}}+x_{3{}}\gamma_{3{}})}\,d\gamma_{1{}}d\gamma_{2{}}d\gamma_{3{}}
\label{eqn160n27}
\end{align}
\[i = 1, 2, 3.\]
Equation \eqref{eqn160n27} is received from the formulas  \eqref{eqn160}--\eqref{eqn162} and $u_{i} \in S$, $E_{}(u_{i}) \in S$, where S is the Schwartz space. 

To make reading easier we copy the formula$\;\;\eqref{eqn201aa}$ at this place:
\begin{equation}\label{eqn160n3g7}
\delta (\gamma_{1{}},\gamma_{2{}},\gamma_{3{}}) = \text{\LARGE{e}}  ^{-\frac{\displaystyle\epsilon^3}{(\gamma_{1{}}^2 + \gamma_{2{}}^2 + \gamma_{3{}}^2)}}\;,\;\; 0 < \epsilon << 1,\;\;\;\epsilon \neq 0.
\end{equation}

The integral operator $E_{}$ maps the space $S$ into itself. It follows from the basic properties of the Fourier transforms of the functions of the space $S$ (e.g.  \cite{GC68} and/or section 3).

\begin{lemma} \label{lm37}
The integral operator $E_{}$ is a linear operator.
\end{lemma}
\begin{proof} 
1) The Schwartz space $S$ is a linear space (obviously).

$\;\;\;\;\;\;\;2) \; E_{}(\lambda_{1}u_{i1{}} + \lambda_{2}u_{i2{}})$ = $\lambda_{1}E_{}(u_{i1{}}) + \lambda_{2}E_{}(u_{i2{}})$

for any $u_{i1{}}, u_{i2{}} \in S$ and any scalars $\lambda_{1}, \lambda_{2}$.

Proposition 2) follows from the properties of the integration operation.

 Q.E.D.
\end{proof}
\begin{theorem} \label{thm227}
The linear integral operator $E_{}$ is defined everywhere in the space S, with values in the space S. The operator $E_{}$ is  bounded for functions of the space S.
\end{theorem}
\begin{proof}Rewrite operator $E_{}$ from formula  \eqref{eqn160n27} in the form
\begin{align}
E_{}(u_{i}) = F^{-1}[A\cdot F[u_{i}]]
\label{eqn1625h7}
\end{align}
\[i = 1, 2, 3.\]
Here $F$ and $F^{-1}$ are a Fourier transform   and a inverse Fourier transform, respectively, $u_{i} \in S$, $E_{}(u_{i}) \in S$.
\begin{equation}\label{eqn201z837}
A = 1 - \delta (\gamma_{1{}},\gamma_{2{}},\gamma_{3{}})
\end{equation}
\[F[u_{i}] = \frac{1}{(2\pi)^{3/2}}\int_{-\infty}^{\infty}\int_{-\infty}^{\infty}
\int_{-\infty}^{\infty}e  ^{i(\tilde x_{1{}}\gamma_{1{}}+\tilde x_{2{}}\gamma_{2{}}
+\tilde x_{3{}}\gamma_{3{}})}\times\]
\[\times u_{i}(\tilde x_{1{}},\tilde x_{2{}}, \tilde x_{3{}})\,d\tilde x_{1{}}d\tilde x_{2{}} d\tilde x_{3{}}\]
The function $A$ of the formula \eqref{eqn201z837} is a continuous, even function of the coordinates $\gamma_{1{}},\gamma_{2{}},\gamma_{3{}}$, that is
\[A(-\gamma_{{}}) = A(\gamma_{}),\;\;\; \vee \; \gamma_{} \in [- \Gamma , \Gamma]\]
and what is more
\begin{equation}\label{eqn160n2n1gh791}
0 < A \leq 1.
\end{equation}
In case if   $A\equiv 1$, from formula \eqref{eqn1625h7}, we have a known result
\[E_{}(u_{i}) \equiv u_{i}\]
If $0 < A < 1\;\;$, A - const., it is evident from the formula \eqref{eqn1625h7} that
\[|E_{}(u_{i})| <  |u_{i}|\]
As known the function $F[u_{i}]\cdot e ^{-i(x_{1{}}\gamma_{1{}}+x_{2{}}\gamma_{2{}}+x_{3{}}\gamma_{3{}})}$ can be represented as the sum of the even and odd functions
\begin{align}
&F[u_{i}]\cdot e ^{-i(x_{1{}}\gamma_{1{}}+x_{2{}}\gamma_{2{}}+x_{3{}}\gamma_{3{}})} = ( F[u_{i}]\cdot e ^{-i(x_{1{}}\gamma_{1{}}+x_{2{}}\gamma_{2{}}+x_{3{}}\gamma_{3{}})})_{even} + 
\nonumber \\
&\quad\quad\quad\quad\quad\quad\quad + (F[u_{i}]\cdot e ^{-i(x_{1{}}\gamma_{1{}}+x_{2{}}\gamma_{2{}}+x_{3{}}\gamma_{3{}})})_{odd}
\label{eqn160n17gl}
\end{align}
Here $( F[u_{i}]\cdot e ^{-i(x_{1{}}\gamma_{1{}}+x_{2{}}\gamma_{2{}}+x_{3{}}\gamma_{3{}})})_{even}$ is the even function in all three coordinates , $(F[u_{i}]\cdot e ^{-i(x_{1{}}\gamma_{1{}}+x_{2{}}\gamma_{2{}}+x_{3{}}\gamma_{3{}})})_{odd}$  is the sum of seven functions, each of which at least one coordinate is an odd.

As known

1. The product of two even functions is even.

2. The product of the even and odd functions is odd.

3. The integral of an odd function by symmetric within is equal zero. The Fourier transform and inverse Fourier transform are the integrals over symmetrical areas.

From the formula \eqref{eqn160n17gl} using the function $A$ we have
\begin{align}
&A F[u_{i}]\cdot e ^{-i(x_{1{}}\gamma_{1{}}+x_{2{}}\gamma_{2{}}+x_{3{}}\gamma_{3{}})} = A( F[u_{i}]\cdot e ^{-i(x_{1{}}\gamma_{1{}}+x_{2{}}\gamma_{2{}}+x_{3{}}\gamma_{3{}})})_{even} + 
\nonumber \\
&\quad\quad\quad\quad\quad\quad\quad + A(F[u_{i}]\cdot e ^{-i(x_{1{}}\gamma_{1{}}+x_{2{}}\gamma_{2{}}+x_{3{}}\gamma_{3{}})})_{odd}
\label{eqn160n17fl}
\end{align}
Since A is the even function, we have

$A( F[u_{i}]\cdot e ^{-i(x_{1{}}\gamma_{1{}}+x_{2{}}\gamma_{2{}}+x_{3{}}\gamma_{3{}})})_{even}$ is the even function (Rule 1)

$A( F[u_{i}]\cdot e ^{-i(x_{1{}}\gamma_{1{}}+x_{2{}}\gamma_{2{}}+x_{3{}}\gamma_{3{}})})_{odd}$ is the odd function (Rule 2)

The inverse Fourier transform $F^{-1}$ is the integral over symmetrical areas.

Then from formula \eqref{eqn1625h7} in accordance with Rule 3 we have
\begin{align}
&\quad\quad\quad\quad\quad\quad\quad E_{}(u_{i}) = F^{-1}[A\cdot F[u_{i}]] =
\nonumber \\
&= \int_{-\infty}^{\infty}\int_{-\infty}^{\infty}\int_{-\infty}^{\infty}A\cdot( F[u_{i}]\cdot e ^{-i(x_{1{}}\gamma_{1{}}+x_{2{}}\gamma_{2{}}+x_{3{}}\gamma_{3{}})})_{even}\,d\gamma_1
d\gamma_2d\gamma_3 \nonumber \\
\label{eqn160n2n1gh79rl}
\end{align}
\[i = 1, 2, 3.\]
In case if  $A\equiv 1$, from formula \eqref{eqn160n2n1gh79rl} we have
\begin{equation}\label{eqn160n2n1gh7jl1l}
E_{}(u_{i}) \equiv u_{i}
\end{equation}
If $0 < A < 1\;\;$, A - const., it is evident from the formula \eqref{eqn160n2n1gh79rl} that
\begin{equation}\label{eqn160n2n1gh7jl2l}
|E_{}(u_{i})| <  |u_{i}|
\end{equation}
We use formula \eqref{eqn160n3g7} for $\delta$ and receive an estimate of the function $A $ from formula \eqref{eqn201z837}:
\begin{equation}\label{eqn201z837o}
0 < A  \leq \epsilon  
\end{equation}
for $(\gamma_{1_{}}^2 + \gamma_{2_{}}^2 + \gamma_{3_{}}^2) \geq  \epsilon^2$ and
\begin{equation}\label{eqn201z837i}
0 << A  \leq 1  
\end{equation}
for $(\gamma_{1_{}}^2 + \gamma_{2_{}}^2 + \gamma_{3_{}}^2) \leq  \epsilon^2.$

 Then we have from formula \eqref{eqn160n2n1gh79rl} using an estimate of function $A $ \eqref{eqn201z837o} and \eqref{eqn201z837i}

\begin{align}
&E_{}(u_{i}) = F^{-1}[A\cdot F[u_{i}]] =
\nonumber \\
&=\int_{R^3 }A\cdot (F[u_{i}] e ^{-i(x_{1{}}\gamma_{1{}}+x_{2{}}\gamma_{2{}}+x_{3{}}\gamma_{3{}})})_{even}\,\,d\gamma_{} =
\nonumber \\
&= \int_{(2\epsilon)^3}A\cdot (F[u_{i}] e ^{-i(x_{1{}}\gamma_{1{}}+x_{2{}}\gamma_{2{}}+x_{3{}}\gamma_{3{}})})_{even}\,\,d\gamma_{}  + 
\nonumber \\
&+\int_{R^3 - (2\epsilon)^3}A\cdot (F[u_{i}] e ^{-i(x_{1{}}\gamma_{1{}}+x_{2{}}\gamma_{2{}}+x_{3{}}\gamma_{3{}})})_{even}\,\,d\gamma_{} =
\nonumber \\
&= I_{(2\epsilon)^3} + I_{R^3 - (2\epsilon)^3}
\nonumber \\
\label{eqn160n2347}
\end{align}
Here we have $u_{i} \in S$ and $F[u_{i}]  \in S$ as well as $|u_{i}|< \infty$ and $|F[u_{i}]|< \infty$  (see formula \eqref{eqn201z9b}). Using estimates for the function A \eqref{eqn201z837o}, \eqref{eqn201z837i}, we have from the formula\eqref{eqn160n2347}
\begin{equation}\label{eqn16297}
I_{(2\epsilon)^3} \sim (2\epsilon)^3,\;\;\;I_{R^3 - (2\epsilon)^3} \sim \epsilon
\end{equation}
We use the formulas \eqref{eqn160n2n1gh7jl1l}, \eqref{eqn160n2n1gh7jl2l} and take estimate of $\epsilon:$ 
\[0 < \epsilon  << 1.\]
Then in accordance with the rules of integration we obtain from formula\eqref{eqn160n2347}
\begin{equation}\label{eqn160n2ng7}
|E_{}(u_{i})| <   {\epsilon}|u_{i}| 
\end{equation}
\[i = 1, 2, 3.\]
i.e. the operator $E_{}$ is bounded for functions of the space S.

 Q.E.D.
\end{proof}

Thus, the Theorem 6.2 implies that the linear integral operator $E_{}$ is bounded for functions of the space S and therefore the matrix integral operator $\bar{\bar{E}}_{}$ is bounded for vector-functions of the space $\overrightarrow{TS}$ and 
\begin{equation}\label{eqn160n4uo7}
|\bar{\bar{E}}_{}\cdot\vec{u}_{}| < {\epsilon} |\vec{u}_{}|, 
\end{equation}
where $\vec{u}_{} \in \overrightarrow{TS}$ and $\bar{\bar{E}}_{}\cdot\vec{u}_{} \in \overrightarrow{TS}$. 

Let us describe  in details the properties of the matrix integral operator $\bar{\bar{S}}_{}$.
\begin{equation}\label{eqn164n37}
\bar{\bar{S}}_{} = 
 \begin{pmatrix}
S_{11{}} & S_{12{}} & S_{13{}} \\
S_{21{}} & S_{22{}} & S_{23{}} \\
S_{31{}} & S_{32{}} & S_{33{}} \end{pmatrix} \quad 
\end{equation}
%(see formula $\;\;\eqref{eqn164m}$).
Here the components $S_{ij{}}$ of the matrix integral operator $\bar{\bar{S}}_{}$ have the following representation:
\begin{align}
&S_{ij{}}(f_{j{}}) = \frac{1}{8\pi^3} \int_0^t\int_{-\infty}^{\infty} \int_{-\infty}^{\infty}
\int_{-\infty}^{\infty} \Big[\chi_{ij}(\gamma_{1{}},\gamma_{2{}},\gamma_{3{}}) e ^{-\nu_{} (\gamma_{1{}}^2 +\gamma_{2{}}^2 +\gamma_{3{}}^2) (t-\tau)}
\nonumber  \\
&\quad\times   \int_{-\infty}^{\infty}\int_{-\infty}^{\infty}
\int_{-\infty}^{\infty}e  ^{i(\tilde x_{1{}}\gamma_{1{}}+\tilde x_{2{}}\gamma_{2{}}
+\tilde x_{3{}}\gamma_{3{}})}
 f_{j{}}(\tilde x_{1{}},\tilde x_{2{}}, \tilde x_{3{}}, \tau)\,d\tilde x_{1{}}d\tilde x_{2{}} d\tilde x_{3{}}\Big] 
\nonumber \\
&\quad\times
\delta (\gamma_{1{}},\gamma_{2{}},\gamma_{3{}}) e ^{-i(x_{1{}}\gamma_{1{}}+x_{2{}}\gamma_{2{}}+x_{3{}}\gamma_{3{}})}\,d\gamma_{1{}}d\gamma_{2{}}d\gamma_{3{}}d\tau
\label{eqn160n37}
\end{align}
Equation \eqref{eqn160n37}  is received from the formulas  \eqref{eqn160}--\eqref{eqn162}.
\begin{gather*}
\\
\chi_{ij}(\gamma_{1{}},\gamma_{2{}},\gamma_{3{}}) :
\\
\frac{( \gamma_{2{}}^2 +\gamma_{3{}}^2)} { (\gamma_{1{}}^2 +\gamma_{2{}}^2+\gamma_{3{}}^2 ) },\quad
\frac{( \gamma_{1{}}\cdot\gamma_{2{}})} { (\gamma_{1{}}^2 +\gamma_{2{}}^2+\gamma_{3{}}^2 ) },\quad
\frac{( \gamma_{1{}}\cdot\gamma_{3{}})} { (\gamma_{1{}}^2 +\gamma_{2{}}^2+\gamma_{3{}}^2 ) }, \\
\frac{( \gamma_{2{}}\cdot\gamma_{1{}})} { (\gamma_{1{}}^2 +\gamma_{2{}}^2+\gamma_{3{}}^2 ) },\quad
\frac{( \gamma_{3{}}^2 +\gamma_{1{}}^2)} { (\gamma_{1{}}^2 +\gamma_{2{}}^2+\gamma_{3{}}^2 ) },\quad
\frac{( \gamma_{2{}}\cdot\gamma_{3{}})} { (\gamma_{1{}}^2 +\gamma_{2{}}^2+\gamma_{3{}}^2 ) },\\
\frac{( \gamma_{3{}}\cdot\gamma_{1{}})} { (\gamma_{1{}}^2 +\gamma_{2{}}^2+\gamma_{3{}}^2 ) }, \quad
\frac{( \gamma_{3{}}\cdot\gamma_{2{}})} { (\gamma_{1{}}^2 +\gamma_{2{}}^2+\gamma_{3{}}^2 ) },\quad
\frac{( \gamma_{1{}}^2 +\gamma_{2{}}^2)} { (\gamma_{1{}}^2 +\gamma_{2{}}^2+\gamma_{3{}}^2 ) }\\
\\
\chi_{ij} = \chi_{ji}, \;\;\;\;\;\;\;\;\;\;\;       i \neq j,\;\;\;\;\;  \;\;\;\;\; \;\;                 i,j = 1,2,3.
\end{gather*}
\begin{equation}\label{eqn160n4kro7}
\end{equation}
\begin{equation}\label{eqn160n4ko7}
\;\;\;\;\;\;\;\;\;\;f_{j{}} =   \sum_{n=1}^3 u_{n{}}\frac{\partial u_{j{}}}{\partial \tilde x_{n{}}},\;\;\;\vec{f}_{} =  (\vec{u}_{}\cdot\nabla_{})\vec{u}{},\\   
\end{equation}
$u_{n{}} \in S$, $\frac{\partial u_{j{}}}{\partial \tilde x_{n{}}} \in S$. $u_{n{}}\frac{\partial u_{j{}}}{\partial \tilde x_{n{}}} \in S$ then  $f_{j{}} \in S$, $S_{ij{}}(f_{j{}}) \in S$, where S is the Schwartz space. 

To make reading easier we copy the formula$\;\;\eqref{eqn201aa}$ at this place:
\begin{equation}\label{eqn160n457}
\delta (\gamma_{1{}},\gamma_{2{}},\gamma_{3{}}) = \text{\LARGE{e}}  ^{-\frac{\displaystyle\epsilon^3}{(\gamma_{1{}}^2 + \gamma_{2{}}^2 + \gamma_{3{}}^2)}}\;,\;\; 0 < \epsilon << 1,\;\;\;\epsilon \neq 0.
\end{equation}

The integral operator $S_{ij{}}$ maps the space S into itself. It follows from the basic properties of the Fourier transforms of the functions of the space S (e.g.  \cite{GC68} and/or section 3).

We solve the equation \eqref{eqn203ab7} for $t \in [0, \delta t]$ with condition  for $\delta t:$

$ 0 <  \Delta t < \delta t  << 1$. Therefore  $ t  << 1$.

$\delta t$ is a very small time increment. $ \Delta t$ is a very small pre-fixed time increment.

For example $\delta t = \text{e}^{-q_3}, \;\;\;q_3 = 2, 3,4,...\;\;\;q_3 < \infty$.

We have the integral operator components $S_{ij{}}$  of the matrix integral operator $\bar{\bar{S}}_{}$ in this case [see formula \eqref{eqn160n37}]:
\begin{align}
&S_{ij{}}(f_{j{}}) = \frac{t}{(2\pi)^{3/2}} \int_{-\infty}^{\infty} \int_{-\infty}^{\infty}
\int_{-\infty}^{\infty} \chi_{ij}(\gamma_{1{}},\gamma_{2{}},\gamma_{3{}}) e ^{-\nu_{} (\gamma_{1{}}^2 +\gamma_{2{}}^2 +\gamma_{3{}}^2) (t-t_*)}
\nonumber  \\
&\quad\times   F_{j{}}(\gamma_{1{}},\gamma_{2{}},\gamma_{3{}}, t_*) \delta (\gamma_{1{}},\gamma_{2{}},\gamma_{3{}})
e ^{-i(x_{1{}}\gamma_{1{}}+x_{2{}}\gamma_{2{}}+x_{3{}}\gamma_{3{}})}\,d\gamma_{1{}}d\gamma_{2{}}d\gamma_{3{}}
\nonumber  \\
&\quad  \equiv tS_{ij{}}^t(f_{j{}})
\label{eqn160nz318}
\end{align}
\[0<  t_*  < t\]
Here
\[F_{j{}}(\gamma_{1{}},\gamma_{2{}},\gamma_{3{}},  t_*) = \frac{1}{(2\pi)^{3/2}}\int_{-\infty}^{\infty}\int_{-\infty}^{\infty}
\int_{-\infty}^{\infty}e  ^{i(\tilde x_{1{}}\gamma_{1{}}+\tilde x_{2{}}\gamma_{2{}}
+\tilde x_{3{}}\gamma_{3{}})}\]
\[ \times f_{j{}}(\tilde x_{1{}},\tilde x_{2{}}, \tilde x_{3{}}, t_*)\,d\tilde x_{1{}}d\tilde x_{2{}} d\tilde x_{3{}}\]

Let us describe in details the properties of the matrix integral operator $\bar{\bar{S}}_{}^{t}$.
\begin{equation}\label{eqn164n3y}
\bar{\bar{S}}_{}^{t} = 
 \begin{pmatrix}
S_{11{}}^{t} & S_{12{}}^{t} & S_{13{}}^{t} \\
S_{21{}}^{t} & S_{22{}}^{t} & S_{23{}}^{t} \\
S_{31{}}^{t} & S_{32{}}^{t} & S_{33{}}^{t} \end{pmatrix} \quad 
\end{equation}
%(see formula $\;\;\eqref{eqn164m}$).
Here the components $S_{ij{}}^{t}$ of the matrix integral operator $\bar{\bar{S}}_{}^{t}$ have the following representation:
\begin{align}
&S^t_{ij{}}(f_{j{}}) = \frac{1}{(2\pi)^{3/2}} \int_{-\infty}^{\infty} \int_{-\infty}^{\infty}
\int_{-\infty}^{\infty} \chi_{ij}(\gamma_{1{}},\gamma_{2{}},\gamma_{3{}}) e ^{-\nu_{} (\gamma_{1{}}^2 +\gamma_{2{}}^2 +\gamma_{3{}}^2) (t-t_*)}
\nonumber  \\
&\quad\times   F_{j{}}(\gamma_{1{}},\gamma_{2{}},\gamma_{3{}}, t_*) \delta (\gamma_{1{}},\gamma_{2{}},\gamma_{3{}})
e ^{-i(x_{1{}}\gamma_{1{}}+x_{2{}}\gamma_{2{}}+x_{3{}}\gamma_{3{}})}\,d\gamma_{1{}}d\gamma_{2{}}d\gamma_{3{}}
\label{eqn160n3y71}
\end{align}
\[0<  t_*  < t\]
Here
\[F_{j{}}(\gamma_{1{}},\gamma_{2{}},\gamma_{3{}},  t_*) = \frac{1}{(2\pi)^{3/2}}\int_{-\infty}^{\infty}\int_{-\infty}^{\infty}
\int_{-\infty}^{\infty}e  ^{i(\tilde x_{1{}}\gamma_{1{}}+\tilde x_{2{}}\gamma_{2{}}
+\tilde x_{3{}}\gamma_{3{}})}\]
\[ \times f_{j{}}(\tilde x_{1{}},\tilde x_{2{}}, \tilde x_{3{}}, t_*)\,d\tilde x_{1{}}d\tilde x_{2{}} d\tilde x_{3{}}\]
Equation \eqref{eqn160n3y71} is received from the formulas  \eqref{eqn160nz318}.
\begin{equation}\label{eqn160n4kro70}
\end{equation}
\begin{gather*}
\chi_{ij}(\gamma_{1{}},\gamma_{2{}},\gamma_{3{}}) :
\\
\frac{( \gamma_{2{}}^2 +\gamma_{3{}}^2)} { (\gamma_{1{}}^2 +\gamma_{2{}}^2+\gamma_{3{}}^2 ) },\quad
\frac{( \gamma_{1{}}\cdot\gamma_{2{}})} { (\gamma_{1{}}^2 +\gamma_{2{}}^2+\gamma_{3{}}^2 ) },\quad
\frac{( \gamma_{1{}}\cdot\gamma_{3{}})} { (\gamma_{1{}}^2 +\gamma_{2{}}^2+\gamma_{3{}}^2 ) }, \\
\frac{( \gamma_{2{}}\cdot\gamma_{1{}})} { (\gamma_{1{}}^2 +\gamma_{2{}}^2+\gamma_{3{}}^2 ) },\quad
\frac{( \gamma_{3{}}^2 +\gamma_{1{}}^2)} { (\gamma_{1{}}^2 +\gamma_{2{}}^2+\gamma_{3{}}^2 ) },\quad
\frac{( \gamma_{2{}}\cdot\gamma_{3{}})} { (\gamma_{1{}}^2 +\gamma_{2{}}^2+\gamma_{3{}}^2 ) },\\
\frac{( \gamma_{3{}}\cdot\gamma_{1{}})} { (\gamma_{1{}}^2 +\gamma_{2{}}^2+\gamma_{3{}}^2 ) }, \quad
\frac{( \gamma_{3{}}\cdot\gamma_{2{}})} { (\gamma_{1{}}^2 +\gamma_{2{}}^2+\gamma_{3{}}^2 ) },\quad
\frac{( \gamma_{1{}}^2 +\gamma_{2{}}^2)} { (\gamma_{1{}}^2 +\gamma_{2{}}^2+\gamma_{3{}}^2 ) }\\
\\
\chi_{ij} = \chi_{ji}, \;\;\;\;\;\;\;\;\;\;\;       i \neq j,\;\;\;\;\;  \;\;\;\;\; \;\;                 i,j = 1,2,3.
\end{gather*}
\begin{equation}\label{eqn160n4kp71}
\;\;\;\;\;\;\;\;\;\;f_{j{}} =   \sum_{n=1}^3 u_{n{}}\frac{\partial u_{j{}}}{\partial \tilde x_{n{}}},\;\;\;\vec{f}_{} =  (\vec{u}_{}\cdot\nabla_{})\vec{u}{},\\
\end{equation}
$u_{n{}} \in S$, $\frac{\partial u_{j{}}}{\partial \tilde x_{n{}}} \in S$. $u_{n{}}\frac{\partial u_{j{}}}{\partial \tilde x_{n{}}} \in S$ then  $f_{j{}} \in S$, $S_{ij{}}^t(f_{j{}}) \in S$, where S is the Schwartz space. 

To make reading easier we copy the formula$\;\;\eqref{eqn201aa}$ at this place:
\begin{equation}\label{eqn160n4571}
\delta (\gamma_{1{}},\gamma_{2{}},\gamma_{3{}}) = \text{\LARGE{e}}  ^{-\frac{\displaystyle\epsilon^3}{(\gamma_{1{}}^2 + \gamma_{2{}}^2 + \gamma_{3{}}^2)}}\;,\;\; 0 < \epsilon << 1,\;\;\;\epsilon \neq 0.
\end{equation}

The integral operator $S_{ij{}}^{t}$ maps the space S into itself. It follows from the basic properties of the Fourier transforms of the functions of the space S (e.g.  \cite{GC68} and/or section 3).
\begin{lemma} \label{lm5y}
The integral operator $S_{ij{}}^{ t}$ is a linear operator.
\end{lemma}
\begin{proof} 
1) The Schwartz space $S$ is a linear space (obviously).

$\;\;\;\;\;\;\;2) \; S_{ij{}}^{t}(\lambda_{1}f_{j1{}} + \lambda_{2}f_{j2{}})$ = $\lambda_{1}S_{ij{}}^{ t}(f_{j1{}}) + \lambda_{2}S_{ij{}}^{ t}(f_{j2{}})$

for any $f_{j1{}}, f_{j2{}} \in S$ and any scalars $\lambda_{1}, \lambda_{2}$.

Proposition 2) follows from the properties of the integration operation.

Q.E.D.
\end{proof}
\begin{theorem} \label{thm32y}
The linear integral operator $S_{ij{}}^{ t}$ is defined everywhere in the space S, with values in the space S. The operator $S_{ij{}}^{ t}$ is bounded for functions of the space S.
\end{theorem}
\begin{proof}
Rewrite operator $S_{ij{}}^{ t}$ from formula \eqref{eqn160n3y71} 
\begin{align}
&S^t_{ij{}}(f_{j{}}) = F^{-1}[A_{ij}\cdot F[f_{j}]]
\label{eqn160nz31hz}
\end{align}
Here $F$ and $F^{-1}$ are a Fourier transform   and a inverse Fourier transform, respectively, $f_{i} \in S$, $S^t_{ij{}}(f_{i}) \in S$.
\begin{equation}\label{eqn160n2n1gh7c}
A_{ij} = \chi_{ij}(\gamma_{1{}},\gamma_{2{}},\gamma_{3{}})e ^{-\nu_{} (\gamma_{1{}}^2 +\gamma_{2{}}^2 +\gamma_{3{}}^2) (t-t^*)}\cdot \delta (\gamma_{1{}},\gamma_{2{}},\gamma_{3{}})
\end{equation}
\[0<  t_*  < t\]
\[F[f_{j}]\equiv F_{j{}}(\gamma_{1{}},\gamma_{2{}},\gamma_{3{}},  t_*) = \]
\[=\frac{1}{(2\pi)^{3/2}}\int_{-\infty}^{\infty}\int_{-\infty}^{\infty}
\int_{-\infty}^{\infty}e  ^{i(\tilde x_{1{}}\gamma_{1{}}+\tilde x_{2{}}\gamma_{2{}}
+\tilde x_{3{}}\gamma_{3{}})}  f_{j{}}(\tilde x_{1{}},\tilde x_{2{}}, \tilde x_{3{}}, t_*)\,d\tilde x_{1{}}d\tilde x_{2{}} d\tilde x_{3{}}\]
In case if   $A_{ij}\equiv 1$, from formula \eqref{eqn160nz31hz}, we have a known result
\[S^t_{ij{}}(f_{j{}})  \equiv f_{j{}}\]
If $0 < A_{ij} < 1\;\;, A_{ij}$ - const., it is evident from the formula \eqref{eqn160nz31hz} that
\[|S^t_{ij{}}(f_{j{}}) | <  |f_{j{}}|\]
Let $i = j.$ The function $A_{ii}$ of the formula \eqref{eqn160n2n1gh7c} is a continuous, even function of the coordinates $\gamma_{1{}},\gamma_{2{}},\gamma_{3{}}$, that is
\[A_{ii}(-\gamma_{{}}) = A_{ii}(\gamma_{}),\;\;\; \vee \; \gamma_{} \in [- \Gamma , \Gamma]\]
(see formula \eqref{eqn160n4kro70} for $\chi_{ii}$ )

And what is more
\begin{equation}\label{eqn160n2n1gh791c}
0 \leq A_{ii} < 1,\;\;\;    i = 1, 2, 3.
\end{equation}
Now consider the function $A_{ij}$ from the formula \eqref{eqn160n2n1gh7c} such that $i \neq j$.

In this case ($i \neq j$) we have obviously from formula \eqref{eqn160n4kro70}:
\begin{equation}\label{eqn160n2n178s}
\chi_{ij}=\frac{( \gamma_{i{}}\cdot\gamma_{j{}})} { (\gamma_{1{}}^2 +\gamma_{2{}}^2+\gamma_{3{}}^2 ) }<\frac{( \gamma_{i{}}^2 +\gamma_{j{}}^2)} { (\gamma_{1{}}^2 +\gamma_{2{}}^2+\gamma_{3{}}^2 ) }=\tilde{\chi}_{ij}=\chi_{kk}
\end{equation}
\[(k \neq i, k \neq j) \]
Then we get from the formula \eqref{eqn160n2n1gh7c}
\begin{align}
& A_{ij} = \chi_{ij}(\gamma_{1{}},\gamma_{2{}},\gamma_{3{}})e ^{-\nu_{} (\gamma_{1{}}^2 +\gamma_{2{}}^2 +\gamma_{3{}}^2) (t-t^*)}\cdot \delta (\gamma_{1{}},\gamma_{2{}},\gamma_{3{}})<
 \nonumber \\
& <  \tilde{\chi}_{ij}(\gamma_{1{}},\gamma_{2{}},\gamma_{3{}})e ^{-\nu_{} (\gamma_{1{}}^2 +\gamma_{2{}}^2 +\gamma_{3{}}^2) (t-t^*)}\cdot \delta (\gamma_{1{}},\gamma_{2{}},\gamma_{3{}}) =
 \nonumber \\
&\quad\quad\quad\quad\quad\quad\quad  = \tilde{A}_{ij}= A_{kk}.\;\;\;\;\;\;\;\;\;(i \neq j, k \neq i, k \neq j) 
\nonumber \\
\label{eqn160n2n1gh79es22}
\end{align}
The function $\tilde{A}_{ij}$ = $A_{kk}$ of the formula \eqref{eqn160n2n1gh79es22} is a continuous, even function of the coordinates $\gamma_{i{}},\gamma_{j{}}, \gamma_{k{}}\;\; $ that is
\[A_{kk}(-\gamma_{{}}) = A_{kk}(\gamma_{}),\;\;\; \vee \; \gamma_{} \in [- \Gamma , \Gamma]\]
(see formula \eqref{eqn160n4kro70} for $\chi_{kk}$ )

And what is more
\begin{equation}\label{eqn160n2n1gh791ccs2}
0 \leq A_{kk} < 1,\;\;\;    k = 1, 2, 3.
\end{equation}
Without loss of generality, we write further that $k = i.$           

In this case we have from formula \eqref{eqn160nz31hz}
\begin{align}
&S^t_{ii{}}(f_{i{}}) = F^{-1}[A_{ii}\cdot F[f_{i}]]
\label{eqn160nz31h1z}
\end{align}
\[i = 1, 2, 3.\]
As known the function $F[f_{i{}}]\cdot e ^{-i(x_{1{}}\gamma_{1{}}+x_{2{}}\gamma_{2{}}+x_{3{}}\gamma_{3{}})}$ can be represented as the sum of the even and odd functions
\begin{align}
&F[f_{i}]\cdot e ^{-i(x_{1{}}\gamma_{1{}}+x_{2{}}\gamma_{2{}}+x_{3{}}\gamma_{3{}})} = ( F[f_{i}]\cdot e ^{-i(x_{1{}}\gamma_{1{}}+x_{2{}}\gamma_{2{}}+x_{3{}}\gamma_{3{}})})_{even} + 
\nonumber \\
&\quad\quad\quad\quad\quad\quad\quad + (F[f_{i}]\cdot e ^{-i(x_{1{}}\gamma_{1{}}+x_{2{}}\gamma_{2{}}+x_{3{}}\gamma_{3{}})})_{odd}
\label{eqn160n17gs1}
\end{align}
Here $( F[f_{i{}}]\cdot e ^{-i(x_{1{}}\gamma_{1{}}+x_{2{}}\gamma_{2{}}+x_{3{}}\gamma_{3{}})})_{even}$ is the even function in all three coordinates , $(F[|f_{i{}}]\cdot e ^{-i(x_{1{}}\gamma_{1{}}+x_{2{}}\gamma_{2{}}+x_{3{}}\gamma_{3{}})})_{odd}$  is  the sum of seven functions, each of which at least one coordinate is an odd.

As known

1. The product of two even functions is even.

2. The product of the even and odd functions is odd.

3. The integral of an odd function by symmetric within is equal zero. The Fourier transform and inverse Fourier transform are the integrals over symmetrical areas.

From the formula \eqref{eqn160n17gs1} using the function $A_{ii}$ we have
\begin{align}
&A_{ii} F[f_{i}]\cdot e ^{-i(x_{1{}}\gamma_{1{}}+x_{2{}}\gamma_{2{}}+x_{3{}}\gamma_{3{}})} = A_{ii}( F[f_{i}]\cdot e ^{-i(x_{1{}}\gamma_{1{}}+x_{2{}}\gamma_{2{}}+x_{3{}}\gamma_{3{}})})_{even} + 
\nonumber \\
&\quad\quad\quad\quad\quad\quad\quad + A_{ii}(F[f_{i}]\cdot e ^{-i(x_{1{}}\gamma_{1{}}+x_{2{}}\gamma_{2{}}+x_{3{}}\gamma_{3{}})})_{odd}
\label{eqn160n17fs1}
\end{align}
Since $A_{ii}$ is the even function, we have

$A_{ii}( F[f_{i}]\cdot e ^{-i(x_{1{}}\gamma_{1{}}+x_{2{}}\gamma_{2{}}+x_{3{}}\gamma_{3{}})})_{even}$ is the even function (Rule 1)

$A_{ii}( F[f_{i}]\cdot e ^{-i(x_{1{}}\gamma_{1{}}+x_{2{}}\gamma_{2{}}+x_{3{}}\gamma_{3{}})})_{odd}$ is the odd function (Rule 2)

The inverse Fourier transform $F^{-1}$ is the integral over symmetrical areas.

Then from formula \eqref{eqn160nz31h1z} in accordance with Rule 3 we have
\begin{align}
&\quad\quad\quad\quad\quad\quad\quad S^t_{ii{}}(f_{i{}}) = F^{-1}[A_{ii}\cdot F[f_{i}]] =
\nonumber \\
&= \int_{-\infty}^{\infty}\int_{-\infty}^{\infty}\int_{-\infty}^{\infty}A_{ii}\cdot( F[f_{i}]\cdot e ^{-i(x_{1{}}\gamma_{1{}}+x_{2{}}\gamma_{2{}}+x_{3{}}\gamma_{3{}})})_{even}\,d\gamma_1
d\gamma_2d\gamma_3 \nonumber \\
\label{eqn160n2n1gh79rs1}
\end{align}
\[i = 1, 2, 3.\]
In case if  $A_{ii}\equiv 1$, from formula \eqref{eqn160n2n1gh79rs1} we have
\begin{equation}\label{eqn160n2n1gh7jl11}
S^t_{ii{}}(f_{i{}}) \equiv f_{i{}}
\end{equation}
If $0 < A_{ii} < 1\;\;$, $A_{ii}$ - const., it is evident from the formula \eqref{eqn160n2n1gh79rs1} that
\begin{equation}\label{eqn160n2n1gh7jl21}
|S^t_{ii{}}(f_{i{}})| <  |f_{i{}}|
\end{equation}
From formula \eqref{eqn160n2n1gh7c} we have
\begin{equation}\label{eqn160n2n1gh7jl10}
0 \leq A_{ii} = \chi_{ii}(\gamma_{1{}},\gamma_{2{}},\gamma_{3{}})e ^{-\nu_{} (\gamma_{1{}}^2 +\gamma_{2{}}^2 +\gamma_{3{}}^2) (t-t^*)}\cdot \delta (\gamma_{1{}},\gamma_{2{}},\gamma_{3{}}) < 1.
\end{equation}
\[0<  t_*  < t\]
The function $A_{ii}$ of the formula \eqref{eqn160n2n1gh7jl10} is a continuous, even function of the coordinates $\gamma_{1{}},\gamma_{2{}},\gamma_{3{}}$.

We use the formulas \eqref{eqn160n2n1gh7jl11}, \eqref{eqn160n2n1gh7jl21} and take the function $A_{ii}$ from formula \eqref{eqn160n2n1gh7jl10}. Then in accordance with the rules of integration we obtain from formula \eqref{eqn160n2n1gh79rs1}
\begin{align}
& |S^t_{ii{}}(f_{i{}})| <  |f_{i{}}|
\nonumber \\
\label{eqn160n2n1gh79e}
\end{align}
\[i = 1, 2, 3.\]
i.e. the operator $S_{ii{}}^{ t}$ is bounded for functions of the space S and  the operator $S_{ij{}}^{ t}$ is bounded for functions of the space S also.

 Q.E.D.
\end{proof}

Thus, the Theorem 6.3 implies that the linear integral operators $S_{ii{}}^{ t}$ and $S_{ij{}}^{ t}$ are bounded for functions of the space S and therefore the matrix integral operator $\bar{\bar{S}}_{}^{ t}$ is bounded for vector-functions of the space $\overrightarrow{TS}$ and, using formula \eqref{eqn160n4kp71}, we have
\begin{equation}\label{eqn203tba}
|\bar{\bar{S}}_{}^{ t}\cdot(\vec{u}_{}\cdot\nabla_{})\vec{u}_{}| < |(\vec{u}_{}\cdot\nabla_{})\vec{u}_{}|, 
\end{equation}
where $(\vec{u}_{}\cdot\nabla_{})\vec{u}_{} \in \overrightarrow{TS}$.

\section{A priori estimate of the solution}

Now we  estimate the dependence of the velocity $\vec{u}_{}\in \overrightarrow{TS}$ from time t. 

Let $t = 0$. Then we have from the equation \eqref{eqn203ab7} by using the properties of the integral operators $\bar{\bar{B_{}}}$, $\bar{\bar{E_{}}}$ and $\bar{\bar{S_{}}}$ (see formulas \eqref{eqn160n17}, \eqref{eqn160n27},\eqref{eqn160nz318}) : 
\begin{equation}\label{eqn203t5}
\bar{\bar{S_{}}} = 0, \;\;\;\vec{u}_{} = \vec{u}^0_{},\;\;\;\;\;\;
\vec{u}^0_{}\in \overrightarrow{TS}
\end{equation}

Let $t = \delta t > 0$. $\delta t$ is a very small time increment.
\begin{equation}\label{eqn203t5ase}
0 <  \Delta t < \delta t  << 1
\end{equation}
Here $ \Delta t$ is a very small pre-fixed time increment.

For example $\delta t = \text{e}^{-q_3}, \;\;\;q_3 = 2, 3,4,...\;\;\;q_3 < \infty$.

Then we may take:
\begin{equation}\label{eqn203f1}
\begin{aligned}
&\vec{u}_{} = \vec{u}^0_{} + \delta\vec{u}_{},\;\;\;\;\;\;
\delta\vec{u}_{} \in \overrightarrow{TS} \\
&\quad\quad\quad\quad|\delta\vec{u}_{}| << |\vec{u}^0_{}|
\end{aligned} \\
\end{equation}
For example $|\delta\vec{u}_{}| = \text{e}^{-q_{4}}|\vec{u}^0_{}|, \;\;\;q_{4} = 2, 3,4,...\;\;\;q_{4} < \infty$.
\begin{equation}\label{eqn203fa1}
\begin{aligned}
&\quad\quad\quad\quad(\vec{u}_{}\cdot\nabla_{})\vec{u}_{} = (\vec{u}^0_{}\cdot\nabla_{})\vec{u}^0_{} + \delta(\vec{u}_{}\cdot\nabla_{})\vec{u}_{},\;\;\;\;\; \\
&(\vec{u}_{}\cdot\nabla_{})\vec{u}_{}  \in \overrightarrow{TS},\;\;\;\;\;(\vec{u}^0_{}\cdot\nabla_{})\vec{u}^0_{}  \in \overrightarrow{TS},\;\;\;\;\;\delta(\vec{u}_{}\cdot\nabla_{})\vec{u}_{}  \in \overrightarrow{TS}.\\
&\quad\quad\quad\quad\quad\quad|\delta(\vec{u}_{}\cdot\nabla_{})\vec{u}_{}| << |(\vec{u}^0_{}\cdot\nabla_{})\vec{u}^0_{}|
\end{aligned} \\
\end{equation}

For example $|\delta(\vec{u}_{}\cdot\nabla_{})\vec{u}_{}| = \text{e}^{-q_{5}}|(\vec{u}^0_{}\cdot\nabla_{})\vec{u}^0_{}|,
\;\;\;q_{5} =  2, 3,4,...\;\;\;q_{5} < \infty$.

Obtain an estimate of the velocity $\vec{u}_{}$ at the moment $t = \delta t > 0$ with conditions \eqref{eqn203t5ase} - \eqref{eqn203fa1}. 

We rewrite the equation \eqref{eqn203ab7} at the moment $t = \delta t > 0$.  
\begin{equation}\label{eqn203t1a}
\vec{u}_{} = - \delta t\bar{\bar{S}}^{\delta t}_{}\cdot(\vec{u}_{}\cdot\nabla_{})\vec{u}_{}+\bar{\bar{E}}_{}\cdot\vec{u}_{}+\bar{\bar{B}}_{}
\cdot\vec{u}^0_{}
\end{equation}
and use the norm \eqref{eqn203t317v} for both sides of the equation \eqref{eqn203t1a}. Then we have:
\begin{align}
&|\vec{u}_{}| = |- \delta t\bar{\bar{S}}^{\delta t}_{}\cdot(\vec{u}_{}\cdot\nabla_{})\vec{u}_{}+\bar{\bar{E}}_{}\cdot\vec{u}_{}+\bar{\bar{B}}_{}
\cdot\vec{u}^0_{}| \leq
\nonumber  \\
&\leq | \delta t\bar{\bar{S}}^{\delta t}_{}\cdot(\vec{u}_{}\cdot\nabla_{})\vec{u}_{}|+|\bar{\bar{E}}_{}\cdot\vec{u}_{}|+|\bar{\bar{B}}_{}
\cdot\vec{u}^0_{}|
\nonumber \\
\label{eqn203t1b}
\end{align}
Using inequalities \eqref{eqn203tba} for operator $\bar{\bar{S}}^{\delta t}_{}$, \eqref{eqn160n4uo7} for operator $\bar{\bar{E}}_{}$ and \eqref{eqn160n2n17b} for operator $\bar{\bar{B}}_{}$, we obtain:
\begin{align}
&|\vec{u}_{}| \leq | \delta t\bar{\bar{S}}^{\delta t}_{}\cdot(\vec{u}_{}\cdot\nabla_{})\vec{u}_{}|+|\bar{\bar{E}}_{}\cdot\vec{u}_{}|+|\bar{\bar{B}}_{}
\cdot\vec{u}^0_{}| \leq
\nonumber  \\
&\leq \delta t|(\vec{u}_{}\cdot\nabla_{})\vec{u}_{}|+{\epsilon} |\vec{u}_{}|+ |\vec{u}^0_{}|
\nonumber \\
\label{eqn203t1bd}
\end{align}
We substitute $\vec{u}_{}$ and $(\vec{u}_{}\cdot\nabla_{})\vec{u}_{}$ from formulas \eqref{eqn203f1} and \eqref{eqn203fa1} in equation \eqref{eqn203t1bd}. Then we have:
\begin{align}
&|\vec{u}_{}| \leq \delta t|(\vec{u}_{}\cdot\nabla_{})\vec{u}_{}|+{\epsilon} |\vec{u}_{}|+ |\vec{u}^0_{}| \leq
\nonumber  \\
&\leq \delta t|(\vec{u}^0_{}\cdot\nabla_{})\vec{u}^0_{}|+\delta t| \delta(\vec{u}_{}\cdot\nabla_{})\vec{u}_{}|+
\nonumber  \\
&+{\epsilon} |\vec{u}^0_{}|+{\epsilon} | \delta\vec{u}_{}|+ |\vec{u}^0_{}|
\nonumber \\
\label{eqn203t1bf}
\end{align}
As $\delta t << 1$, ${\epsilon} << 1$ and $|\delta(\vec{u}_{}\cdot\nabla_{})\vec{u}_{}| << |(\vec{u}^0_{}\cdot\nabla_{})\vec{u}^0_{}|$ (see formula \eqref{eqn203fa1}),  

$|\delta\vec{u}_{}| << |\vec{u}^0_{}|$ (see formula \eqref{eqn203f1}) we neglect small terms of the second order 

$\delta t| \delta(\vec{u}_{}\cdot\nabla_{})\vec{u}_{}|$ and ${\epsilon} | \delta\vec{u}_{}|$ and obtain:
\begin{align}
&|\vec{u}_{}| \leq \delta t|(\vec{u}^0_{}\cdot\nabla_{})\vec{u}^0_{}|+{\epsilon} |\vec{u}^0_{}|+ |\vec{u}^0_{}|
\nonumber \\
\label{eqn203t1bg}
\end{align}
Since $\vec{u}^0_{} \in\overrightarrow{TS}$ and $(\vec{u}^0_{}\cdot\nabla_{})\vec{u}^0_{} \in\overrightarrow{TS}$ then $0 < |\vec{u}^0_{}| < C^0 < \infty$ and 

$0 < |(\vec{u}^0_{}\cdot\nabla_{})\vec{u}^0_{}|< C^{\nabla0} < \infty$. $\;\;C^0, C^{\nabla0}$ are consts. 

We can introduce $\widetilde{\delta t} << 1$  
\[\delta t  = \widetilde{\delta t} \frac{|\vec{u}^0_{}|}{|(\vec{u}^0_{}\cdot\nabla_{})\vec{u}^0_{}|}\]
and we have from formula \eqref{eqn203t1bg}:
\begin{align}
&|\vec{u}_{}| \leq \widetilde{\delta t}|\vec{u}^0_{}|+{\epsilon} |\vec{u}^0_{}|+ |\vec{u}^0_{}| =
\nonumber \\
&\quad\quad = (\widetilde{\delta t}+{\epsilon}+1)  |\vec{u}^0_{}|
\nonumber \\
\label{eqn203t1bg6}
\end{align}
As $\widetilde{\delta t }<< 1$ and ${\epsilon} << 1$, we neglect terms of the order of smallness $\widetilde{\delta t}|\vec{u}^0_{}|$ and terms of the order of smallness ${\epsilon} |\vec{u}^0_{}|$ as compared with $ |\vec{u}^0_{}|$.
We have the evaluation for velocity $\vec{u}_{}$ at time $\delta t$ from the equation \eqref{eqn203t1bg6}:
\begin{equation}\label{eqn203t3a}
|\vec{u}_{}|  \leq  |\vec{u}^0_{}|
\end{equation}
\begin{remark} \rm
Further, repeating the arguments of evaluation of the Cauchy problem solution
for the Navier Stokes equations \eqref{eqn203t5} - \eqref{eqn203t3a} with initial time $t = \delta t$ instead $t = 0$
and the initial velocity $\vec{u}_{}$ instead $\vec{u}^0_{}$, we again obtain a decrease of rate of velocity $\vec{u}_{}$ for the next small interval of time $\delta t$. These arguments and equations \eqref{eqn203t5} - \eqref{eqn203t3a} can be repeated arbitrarily long. Thus, assessing the nature of the behavior of velocity $\vec{u}_{}$ over time, we see that the rate of velocity $\vec{u}_{}$ decreases monotonically over time.
\end{remark}
It should be noted that this estimate is obtained under the conditions \eqref{eqn203t5ase} - \eqref{eqn203fa1}. 

\section{The solution for 3D Navier-Stokes equations with any smooth initial velocity.}

Let us rewrite the integral equation \eqref{eqn203ab7} for $\;t \in [0,\delta t]. $ 

\begin{equation}\label{eqn203ab77}
\vec{u}_{}=- t \bar{\bar{S}}^t_{}\cdot(\vec{u}_{}\cdot\nabla_{})\vec{u}_{}
+\bar{\bar{E}}_{}\cdot\vec{u}_{}
+\bar{\bar{B}}_{}\cdot\vec{u}_{}^0.
\end{equation}
$ 0 <  \Delta t < \delta t  << 1$. Therefore  $ t  << 1$.

For example $\delta t = \text{e}^{-q_3}, \;\;\;q_3 = 2, 3,4,...\;\;\;q_3 < \infty$.

Here $\vec{u}_{} \in \overrightarrow{TS},\;\;$ $(\vec{u}_{}\cdot\nabla_{})\vec{u}_{} \in \overrightarrow{TS},\;\;$ $\bar{\bar{S}}^t_{}\cdot(\vec{u}_{}\cdot\nabla_{})\vec{u}_{} \in \overrightarrow{TS},\;\;$ $\bar{\bar{E}}_{}\cdot\vec{u}_{} \in \overrightarrow{TS},\;\;$ $\vec{u}^0_{} \in \overrightarrow{TS},\;\;$  $\bar{\bar{B}}_{}\cdot\vec{u}_{}^0 \in \overrightarrow{TS}$. 

Operators $\bar{\bar{S}}^t_{},\;\;$ $\bar{\bar{E}}_{}\;$ and $\;\bar{\bar{B}}_{}$ are bounded for vector-functions of the space $\overrightarrow{TS}$.

\begin{theorem} \label{thm31ym}
There exists the solution $\vec{u}$
%eqn203t1a
of the equation \eqref{eqn203ab77} in
the space $\overrightarrow{TS}$ for any time $t \in [0, \delta t]$.
\end{theorem}
\begin{proof}
We rewrite the integral equation \eqref{eqn203ab77} for $\;t \in [0,\delta t]. $ 
\begin{equation}\label{eqn203ab77az}
\vec{u}_{}=- t \bar{\bar{S}}^t_{}\cdot(\vec{u}_{}\cdot\nabla_{})\vec{u}_{}
+\bar{\bar{E}}_{}\cdot\vec{u}_{}
+\bar{\bar{B}}_{}\cdot\vec{u}_{}^0.
\end{equation}
$ 0 <  \Delta t < \delta t  << 1$. Therefore  $ t  << 1$.

We have $\vec{u}^0_{} \in \overrightarrow{TS}$. Then $\bar{\bar{B}}_{}\cdot\vec{u}^0_{} \in \overrightarrow{TS}$  due to the properties of the operator $\bar{\bar{B}}_{}$ (see formulas \eqref{eqn164n7} - \eqref{eqn160n2n17b} and the basic properties of the Fourier transforms of the functions of the space S (section 3)).

Let us  assume that $\vec{u}_{} \in \overrightarrow{TS}$. Thereat  $(\vec{u}_{}\cdot\nabla_{})\vec{u}_{} \in \overrightarrow{TS}\;\;$ due to the properties of the space $\overrightarrow{TS}$. 

Then $\bar{\bar{S}}^t_{}\cdot(\vec{u}_{}\cdot\nabla_{})\vec{u}_{} \in \overrightarrow{TS},\;\;$ due to the properties of the operator $\bar{\bar{S}}^t_{}$ (see formulas \eqref{eqn164n37} - \eqref{eqn203tba} and the basic properties of the Fourier transforms of the functions of the space S (section 3)).

Furthermore $\bar{\bar{E}}_{}\cdot\vec{u}_{} \in \overrightarrow{TS},\;\;$due to the properties of the operator $\bar{\bar{E}}_{}$ (see formulas \eqref{eqn164n27} - \eqref{eqn160n4uo7} and the basic properties of the Fourier transforms of the functions of the space S (section 3)).

Owing to all this  it is evident that a solution of the equation \eqref{eqn203ab77} for any time $t \in [0, \delta t]$  is $\vec{u}_{} \in \overrightarrow{TS}$.

Q.E.D.
\end{proof}
\begin{theorem} \label{thm31ym}
There exists the unique solution $\vec{u}$
%eqn203t1a
of the equation \eqref{eqn203ab77} in
the space $\overrightarrow{TS}$ for any time $t \in [0, \delta t]$ .
\end{theorem}
\begin{proof}
We rewrite the integral equation \eqref{eqn203ab77} for $\;t \in [0,\delta t]. $ 
\begin{equation}\label{eqn203ab77a}
\vec{u}_{}=- t \bar{\bar{S}}^t_{}\cdot(\vec{u}_{}\cdot\nabla_{})\vec{u}_{}
+\bar{\bar{E}}_{}\cdot\vec{u}_{}
+\bar{\bar{B}}_{}\cdot\vec{u}_{}^0.
\end{equation}
$ 0 <  \Delta t < \delta t  << 1$. Therefore  $ t  << 1$.

Let us  assume that the opposite is true.
Then there exist $\vec{u}_{},\vec{u}'_{} \in \overrightarrow{TS}$ are different solutions of the equation \eqref{eqn203ab77a}. 

We introduce
\[\Delta \vec{u}_{} = \vec{u}_{} - \vec{u}'_{}\]
where $\Delta \vec{u}_{} \in \overrightarrow{TS}.$ Obviously $\Delta \vec{u}^0_{} = 0.$ Here  $\Delta \vec{u}^0_{}$ is an initial velocity for this case.

Further we repeat the calculation \eqref{eqn203t5} - \eqref{eqn203t3a} in this case for any time $t \in [0, \delta t]$ and receive an inequality, analogous \eqref{eqn203t3a}.
\begin{equation}\label{eqn203t3zg}
|\Delta\vec{u}_{}|  \leq  |\Delta\vec{u}^0_{}| = 0.
\end{equation}
Therefore 
\begin{equation}\label{eqn203t3zgk}
|\Delta\vec{u}_{}| = 0.  
\end{equation}
Thus, there exists a unique solution $\vec{u}_{} \in \overrightarrow{TS}$ of the equation \eqref{eqn203ab77} for $\;t \in [0,\delta t]. $ 

Q.E.D.
\end{proof}
Then vector-function $\nabla p \in \overrightarrow{TS}$
is defined by \eqref{eqn1} where vector-function $\vec{u}$ is received
from equation \eqref{eqn203ab77}. Function $p$ is defined up to an arbitrary
constant.

Further, repeating the arguments of the Cauchy problem solution for the Navier Stokes equations \eqref{eqn203t5} - \eqref{eqn203t3zgk} with initial time $t = \delta t$ instead of $t = 0$ and the initial velocity $\vec{u}|_{t=\delta t}$ instead of $\vec{u}^0$, we again obtain an estimate of velocity $\vec{u}$ for the next small interval of time $\delta t$ (For example $\delta t = \text{e}^{-q_3}, \;\;\;q_3 = 2, 3,4,...\;\;\;q_3 < \infty$.)
and then the solution $\vec{u}$ for this interval of time $\delta t$. These arguments and equations  (\eqref{eqn203t5} - \eqref{eqn203t3zgk}) can be repeated arbitrarily long.  Availability $\Delta t$  leads to the fact that the process described by equations (\eqref{eqn203t5} - \eqref{eqn203t3zgk}) continue for $t \rightarrow \infty$. 
\begin{remark} \rm
From the above statements, it follows that 
there exists the unique set of smooth functions
  $u_{\infty i}(x, t)$, $p_{\infty}(x, t)$  $(i = 1, 2, 3)$
$\mathbb{R}^3 \times [0,\infty)$ that satisfies \eqref{eqn1}, \eqref{eqn2},
\eqref{eqn3} and
\begin{equation}\label{eqn186b}
u_{\infty i},p_{\infty} \in  C^{\infty}(\mathbb{R}^3 \times [0,\infty)),
\end{equation}
Then, using the inequality $\|{\vec{u}}\|_{L_2}\leq \|{\vec{u^0}}\|_{L_2}$
from  \cite{oL69}, \cite{LK63}, we have
\begin{equation}\label{eqn186c}
\int_{\mathbb{R}^3}|\vec{u}_{\infty}(x, t)|^2dx < C,\quad \forall t\geq 0.
\end{equation}
Let us consider $\nu \to$ 0.
Then we see that inequalities  \eqref{eqn160n2n1gh7}, \eqref{eqn160n2n1gh7c}  are correct also in case 
of Euler equations;
 i.e., there exists unique smooth solution in all time range for this case.
\end{remark}

Hence, we can see that when velocity ${\vec{u}^0}\in \overrightarrow{TS}$, the fluid flow is laminar.
Turbulent flow may occur when velocity ${\vec{u}^0}\notin\overrightarrow{TS}$.

\section{Appendix}

The Fourier integral can be stated in the forms:
\begin{equation}\label{A3}
\begin{gathered}
\begin{aligned}
U( \gamma_1 , \gamma_2 , \gamma_3)
&=F[ u(x_1 , x_2 , x_3)]\\
&= \frac{1}{(2\pi)^{3/2}}
\int_{-\infty}^{\infty} \int_{-\infty}^{\infty}
\int_{-\infty}^{\infty} u( x_1 , x_2 , x_3)
e  ^{ i( \gamma_1 x_1 + \gamma_2 x_2 + \gamma_3 x_3) } dx_1 dx_2 dx_3
\end{aligned}\\
u( x_1 , x_2 , x_3)= \frac{1}{(2\pi)^{3/2}} \int_{-\infty}^{\infty} \int_{-\infty}^{\infty} \int_{-\infty}^{\infty} U( \gamma_1 , \gamma_2,\gamma_3  )\, e  ^{- i( \gamma_1 x_1 + \gamma_2 x_2 + \gamma_3 x_3) } d\gamma_1 d\gamma_2 d\gamma_3
\end{gathered}
\end{equation}
The Laplace integral is usually stated in the form
\begin{equation}\label{A4}
U^{\otimes}(\eta)=L[u(t)]= \int_0^{\infty}u(t) e  ^{-\eta t}dt
\quad u(t)=\frac{1}{2\pi i}\int_{c- i \infty }^{c + i \infty} U^{\otimes}(\eta)
 e  ^{\eta t}d\eta \quad c > c_0.
\end{equation}
Then
\begin{equation}\label{A5}
L[u'(t)]=\eta U^{\otimes}(\eta)-u(0).
\end{equation}
The convolution theorem \cite{DP65,DW46} is stated as:
If integrals
\[
U_1^{\otimes}(\eta)= \int_0^{\infty}u_1(t)
e  ^{-\eta t}d\,t  \quad
 U_2^{\otimes}(\eta)= \int_0^{\infty}u_2(t) e  ^{-\eta t}d\,t
\]
converge absolutely for $\operatorname{Re} \eta > \sigma_{d}$,
 then  $U^{\otimes}(\eta)=U_1^{\otimes}(\eta) U_2^{\otimes}(\eta)$
is Laplace transform of
\begin{equation}\label{A6}
u(t)=\int_0^tu_1(t-\tau)\,u_2(\tau)\,d\,\tau
\end{equation}
A useful Laplace integral is
\begin{equation}\label{A7}
L[e  ^{\eta_kt}]=\int_0^{\infty}e  ^{-(\eta-\eta_k)\,t}d\,t
= \frac{1}{(\eta-\eta_k)}\quad \operatorname{Re}\eta >\eta_k
\end{equation}

\subsection*{Acknowledgments}
 We express our sincere gratitude to Professor L. Nirenberg, at whose suggestion this study was carried out, 
and to Professor Ya.G. Sinai, who found that our article is extremely interesting. 
We are also very thankful to Professor A. B. Gorstko for helpful friendly discussions.

\end{document}